\documentclass[11pt]{article}

\usepackage{microtype}
\usepackage{graphicx}
\usepackage{booktabs} 

\usepackage{hyperref}

\usepackage{amsmath,amsbsy,amsgen,amscd,amssymb,amsthm,amsfonts,stmaryrd}
\usepackage{mathtools}

\usepackage{bm}

\usepackage{microtype}      

\usepackage{hyperref}      
\usepackage{url}            

\usepackage[usenames,dvipsnames]{xcolor}
\usepackage{graphicx}

\graphicspath{{art/}}

\definecolor{dark-gray}{gray}{0.3}
\definecolor{dkgray}{rgb}{.4,.4,.4}
\definecolor{dkblue}{rgb}{0,0,.5}
\definecolor{medblue}{rgb}{0,0,.75}
\definecolor{rust}{rgb}{0.5,0.1,0.1}

\hypersetup{urlcolor=rust}
\hypersetup{citecolor=dkblue}
\hypersetup{linkcolor=dkblue}

\theoremstyle{definition}

\numberwithin{equation}{section}

\renewcommand{\phi}{\varphi}

\DeclareFontFamily{OT1}{pzc}{}
\DeclareFontShape{OT1}{pzc}{m}{it}{<-> s * [1.200] pzcmi7t}{}
\DeclareMathAlphabet{\mathpzc}{OT1}{pzc}{m}{it}

\usepackage{natbib}
\usepackage{algorithm}
\usepackage[margin=1in]{geometry}

\renewcommand{\cite}[1]{\citep{#1}}

\usepackage[nameinlink]{cleveref}
\crefname{lemma}{lemma}{lemmas}

\usepackage{graphicx}
\usepackage{subcaption}

\usepackage{amsmath}
\usepackage{amsthm}

\usepackage{algorithm}
\usepackage{algpseudocode}
\algnewcommand\algorithmicinput{\textbf{Input:}}
\algnewcommand\Input{\item[\algorithmicinput]}
\algdef{SE}[SUBALG]{Indent}{EndIndent}{}{\algorithmicend\ }%
\algtext*{Indent}
\algtext*{EndIndent}

\newtheorem{theorem}{Theorem}
\newtheorem{lemma}{Lemma}

\newtheorem{remark}{Remark}

\makeatletter
\newcommand{\newreptheorem}[2]{\newtheorem*{rep@#1}{\rep@title}
	\newenvironment{rep#1}[1]{\def\rep@title{#2 \ref*{##1}}\begin{rep@#1}}{\end{rep@#1}}
}
\makeatother

\newreptheorem{lemma}{Lemma}
\newreptheorem{theorem}{Theorem}
\newreptheorem{claim}{Claim}
\newreptheorem{proposition}{Proposition}
\newreptheorem{corollary}{Corollary}

\newcommand{\bc}[1]{\left\{{#1}\right\}}
\newcommand{\br}[1]{\left({#1}\right)}

\newcommand{\ip}[2]{\left\langle{#1},{#2}\right\rangle}
\newcommand{\norm}[1]{\left\| {#1} \right\|}

\newcommand{\gradf}[1]{\nabla f \left( {#1} \right)}

\newcommand{\reals}{\mathbb{R}}
\newcommand{\K}{\mathcal{K}}

\newcommand{\xbar}[1]{\bar{x}_{{#1}}}
\newcommand{\ztilde}[1]{\tilde{z}_{{#1}}}

\newcommand{\tg}{\tilde{g}}
\newcommand{\tM}{\tilde{M}}
\newcommand{\teta}{\tilde{\eta}}

\title{UniXGrad: A Universal, Adaptive Algorithm with Optimal Guarantees for Constrained Optimization}

\author{%
  Ali Kavis\thanks{Equal contribution}\\
  EPFL\\
  \texttt{ali.kavis@epfl.ch}\\
  \and
  Kfir Y. Levy\footnotemark[1]\\
  Technion\\
  \texttt{kfirylevy@technion.ac.il}\\
  \\\and
  Francis Bach\\
  INRIA\\
  \texttt{francis.bach@inria.fr}\\
  \and
  Volkan Cevher\\
  EPFL\\
  \texttt{volkan.cevher@epfl.ch}\\
}

\begin{document}

\maketitle

%%%%%%%%%%%%%%%%%%%%%%%%%%%%%%%%%%%%%%%%%%%%%%%%%%%%%%%%%%%%%%%%%%%%%%%%%%%%%%%%%%%%%%%%%%%%%%
%%%%%%%%%%%%%%%%%%%%%%%%%%%%%%%%%%%%%%%%%%%%%%%%%%%%%%%%%%%%%%%%%%%%%%%%%%%%%%%%%%%%%%%%%%%%%%
%!TEX root = main.tex
%%%%%%%%%%%%%%%%%%%%%%%%%%%%%%%%%%%%%%%%%%%%%%%%%%%%%%%%%%%%%%%%%%%%%%%%%%%%%%%%%%%%%%%%%%%%%%
%%%%%%%%%%%%%%%%%%%%%%%%%%%%%%%%%%%%%%%%%%%%%%%%%%%%%%%%%%%%%%%%%%%%%%%%%%%%%%%%%%%%%%%%%%%%%%

\begin{abstract}
	We propose a novel adaptive, accelerated algorithm for the stochastic constrained convex optimization setting.
Our method, which is inspired by the Mirror-Prox method,  \emph{simultaneously} achieves the optimal rates for smooth/non-smooth problems with either deterministic/stochastic first-order oracles. This is done without any prior knowledge of the smoothness nor the noise properties of the problem.
% Our results directly apply to constrained problems with simple set constraints. We initially provide weighted regret bounds for our algorithm in the aforementioned settings, and we propose a novel, yet very clean, online to offline conversion scheme that yields $\mathcal O (\frac{1}{\sqrt{T}})$ rate for non-smooth objective, $\mathcal O (\frac{1}{T^2})$ and $\mathcal O (\frac{1}{T^2} + \frac{\sigma}{\sqrt{T}})$ for smooth objectives with deterministic and stochastic oracles, respectively, where $\sigma^2$ is the variance of unbiased gradients. 
% Our framework is universal with an adaptive learning rate such that it accomplishes the optimal convergence guarantees without any prior knowledge of smoothness, Lipschitz constant and variance of the oracles.
  To the best of our knowledge, this is the first adaptive, unified algorithm that achieves the optimal rates in the constrained setting. We demonstrate the practical performance of our framework through extensive numerical experiments.
\end{abstract}

%%%%%%%%%%%%%%%%%%%%%%%%%%%%%%%%%%%%%%%%%%%%%%%%%%%%%%%%%%%%%%%%%%%%%%%%%%%%%%%%%%%%%%%%%%%%%%
%%%%%%%%%%%%%%%%%%%%%%%%%%%%%%%%%%%%%%%%%%%%%%%%%%%%%%%%%%%%%%%%%%%%%%%%%%%%%%%%%%%%%%%%%%%%%%
%!TEX root = main.tex
%%%%%%%%%%%%%%%%%%%%%%%%%%%%%%%%%%%%%%%%%%%%%%%%%%%%%%%%%%%%%%%%%%%%%%%%%%%%%%%%%%%%%%%%%%%%%%
%%%%%%%%%%%%%%%%%%%%%%%%%%%%%%%%%%%%%%%%%%%%%%%%%%%%%%%%%%%%%%%%%%%%%%%%%%%%%%%%%%%%%%%%%%%%%%

\section{Introduction}
Stochastic constrained optimization with first-order oracles (SCO) is critical in machine learning. Indeed, the scalability of  classical machine learning tasks, such as support vector machines (SVMs), linear/logistic regression and Lasso, rely on efficient  \textit{stochastic} optimization methods. Importantly,  generalization guarantees for such tasks often rely on constraining  the set of possible solutions. The latter  induces simple solutions in the form of low norm or low entropy, which in 
trun enables to establish generalization guarantees.
%while the generalization guarantees of these tasks often depend on the bounds on their parameter \textit{constraints}. 

In the SCO setting, the optimal convergence rates for the cases of non-smooth and smooth objectives are given by $\mathcal{O}(GD/\sqrt{T})$ and $\mathcal{O}(LD^2/T^2 + \sigma D  /\sqrt{T})$, respectively; where $T$ is the total number of (noisy) gradient queries, $L$ is the smoothness constant of the objective,  $\sigma^2$ is the variance of the stochastic gradient estimates, $D$ is the effective diameter of the decision set, and $G$ is a bound on the magnitude of gradient estimates. These rates cannot be improved without additional assumptions.    

The optimal rate for the non-smooth case may be obtained by the current state-of-the-art  optimization algorithms, such as Stochastic Gradient Descent (SGD), AdaGrad \citep{duchi2011adaptive}, Adam \citep{kingma2014adam}, and AmsGrad \citep{j.2018on}. However, in order to obtain the optimal  rate for the smooth case, one is required to use more involved \emph{accelerated} methods such as \citep{hu2009accelerated,lan2012optimal,xiao2010dual,diakonikolas2017accelerated,cohen2018acceleration,deng2018optimal}.  

Unfortunately, all of these accelerated methods require  a-priori knowledge of the smoothness parameter $L$, as well as the variance of the gradients $\sigma^2$, creating a setup barrier for their use in practice. As a result,  accelerated methods are not very popular in machine learning tasks.

This work develops a new \emph{universal} method for SCO that obtains the optimal rates in both smooth and non-smooth cases, \textit{without any prior knowledge regarding the smoothness of the problem $L$, nor the noise magnitude $\sigma$}.  
Such universal methods that implicitly adapt to the properties of the learning objective may be very beneficial in practical large-scale problems where these properties are usually unknown. To our knowledge, this is the first work that achieves this desiderata in the constrained SCO setting.

\paragraph{Our contributions in the context of related work}

For the unconstrained setting, \citet{levy2018online} and \citet{cutkosky2019anytimeICML} have recently presented a universal scheme that  obtains (almost) optimal rates for both smooth and non-smooth cases. 

More specifically, \citet{levy2018online} designs AcceleGrad---a method that obtains respective rates of $\mathcal{O}\left(G D \sqrt{\log T}/ \sqrt{T}\right)$ and $\mathcal{O}\left(L \log{L} D^2 / T + \sigma D \sqrt{\log{T}}/\sqrt{T} \right)$. %, which is done without any prior knowledge of the smoothness nor the noise variance.
Unfortunately, this result only holds for the unconstrained  setting, and the authors leave the \textit{constrained} case as an open problem.

An important progress towards this open problem is achieved only recently by \citet{cutkosky2019anytimeICML}, who proves suboptimal  respective rates  of $\mathcal{O}\left(1/\sqrt{T}\right)$ and $\mathcal{O}\left(D^2 L/T^{3/2} + \sigma D /\sqrt{T}\right)$ for SCO in the constrained setting. %We note that the setting considered by \citet{cutkosky2019anytimeICML} is online optimization which appears more general than SCOFO. 

Our work completely resolves the open problem in \citet{levy2018online,cutkosky2019anytimeICML}, and proposes the first \textit{universal} method that obtains respective \emph{optimal} rates of  $ \mathcal{O}\left(G D /\sqrt{T}\right)$ and $ \mathcal{O} \left(D^2 L/T^{2} + \sigma D /\sqrt{T}\right)$ for the constrained setting. When applied to the unconstrained setting, our analysis tightens the rate characterizations by removing the unnecessary logarithmic factors appearing in \citep{levy2018online,cutkosky2019anytimeICML}.

Our method is inspired by the Mirror-Prox method \citep{nemirovski2004prox,rakhlin2013optimization,diakonikolas2017accelerated,bach2019universal}, and builds on top of it using additional techniques from the online learning literature.
Among, is an adaptive learning rate rule \citep{duchi2011adaptive,rakhlin2013optimization}, as well as recent  online-to-batch conversion techniques \citep{levy2017online,cutkosky2019anytimeICML}. 

The paper is organized as follows. In the next section, we specify the problem setup, and give the necessary definitions and background information. In Section~\ref{sec:algorithm}, we motivate our framework and explain the general mechanism. We also introduce the convergence theorems with proof sketches to highlight the technical novelties. We share numerical results in comparison with other adaptive methods and baselines for different machine learning tasks in Section~\ref{sec:experiments}, followed up with conclusions.

\section{Setting and preliminaries} \label{sec:prelim}

\paragraph{Preliminaries.}
 Let $\|\cdot\|$ be a general norm and $\|\cdot\|_*$ be its dual norm. 
A  function $f:\K\mapsto\reals$ is \emph{$\mu$-strongly convex} over a convex set $\K$, if for any $x\in\K$ and 
 any $\nabla f(x)$, a subgradient of $f$ at~$x$, 
\begin{align} \label{prelim:str-cvx}
	f(x) - f(y) - \ip{\nabla f (y)}{x - y} \geq \frac{\mu}{2} \norm{x - y}^2, \quad \forall x,y \in \mathcal K
\end{align}
%for any choice $\nabla f(x)$ of subgradient of $f$ at $x$.
A  function $f:\K\mapsto\reals$ is \emph{$L$-smooth} over  $\K$ if it has $L$-Lipschitz continuous gradient, i.e.,
\vspace{-3mm}

\begin{align} \label{prelim:L-smooth}
	\norm{\nabla f(x) - \nabla f(y)}_\ast \leq L \norm{x - y}, \quad \forall x,y \in \mathcal K.
\end{align}

Consider a $1$-strongly convex differentiable function $\mathcal R: \mathcal K \rightarrow \mathbb R$. The Bregman divergence with respect to a distance-generating function $\mathcal R$ is defined as follows $\forall x,y\in\K$,
%\vspace{-3mm}
\begin{align}
	D_{\mathcal R} (x, y) = \mathcal R(x) - \mathcal R(y) - \ip{\nabla \mathcal R (y)}{x - y}~.
\end{align}
An important property of Bregman divergence is that $D_{\mathcal R} (x, y) \geq \frac{1}{2} \norm{x - y}^2$ for all $x,y\in\K$, due to the strong convexity of $\mathcal R$.

\paragraph{Setting}
This paper focuses on (approximately) solving the following  constrained  problem,
%\vspace{-3mm}
\begin{align} \label{problem-def}
	\min\limits_{x \in \mathcal K} f(x)~,
\end{align}
where $f:\K\mapsto \reals$ is a convex function, and $\K \subset \reals^d$ is a compact  convex set. 

We assume the availability of a first order oracle for $f(\cdot)$, and consider two settings: a deterministic setting where we may access  exact gradients, and a stochastic setting where we may only access unbiased (noisy) gradient estimates. Concretely, we assume that by querying this oracle with a point $x\in \K$, we receive $\tilde{\nabla} f(x)\in\reals^d$ such,
\begin{align} \label{def:stochastic-gradient}
	\mathbb E \left[ \tilde{\nabla} f (x) \big| x \right] = \nabla f(x)~.
\end{align} 
Throughout this paper we also assume the norm of the (sub)-gradient estimates is bounded by $G$, i.e,
$$
\|\tilde{\nabla}f(x) \|_\ast \leq G,\qquad \forall x\in\K~.
$$

\section{The algorithm} \label{sec:algorithm}

In this section, we present and analyze our \textbf{Uni}versal e\textbf{X}tra \textbf{Grad}ient (UniXGrad) method. We first discuss the Mirror-Prox (MP) algorithm of \citep{nemirovski2004prox}, and the  related Optimistic Mirror Descent (OMD) algorithm of \citep{rakhlin2013optimization}. Later we present our algorithm which builds on top of the Optimistic Mirror Descent (OMD) scheme. Then in Sections~\ref{sec:nonsmooth-case}  and \ref{sec:smooth-case}, we present and analyze the guarantees of our method in nonsmooth and smooth settings, respectively.

%
%\paragraph{Online Optimistic Mirror Prox}
%In the context of the online learning setting \citep{rakhlin2013optimization} introduce the template appearing in Algorithm~\ref{alg:template}. Assuming that $\{g_t\}^_{t=1}^T$ is a set of linear 
%\newpage

Our goal is to optimize a convex function $f$ over a compact domain $\mathcal K$, and Algorithm~\ref{alg:template} offers a framework for solving this template, which is inspired by the Mirror-Prox (MP) algorithm of \citep{nemirovski2004prox} and the Optimistic Mirror Descent (OMD) algorithm of \citep{rakhlin2013optimization}. Let us motivate this particular template. Basically, the algorithm takes a step from $y_{t-1}$ to $x_t$, using first order information based on $y_{t-1}$. Then, it goes back to $y_{t-1}$ and takes another step, but this time, gradient information relies on $x_t$. Each step is a generalized projection with respect to Bregman divergence $D_{\mathcal R}(\cdot, \cdot)$.

\begin{algorithm}[H]
\caption{Mirror-Prox Template} \label{alg:template}
\begin{algorithmic}[1]
\Input{ Number of iterations $T$ , $y_0 \in \mathcal K$, learning rate $\bc{\eta_t}_{t \in [T]}$}
\For{$t = 1, ..., T$}
	\Indent
		\State $x_t = \arg \min\limits_{x \in \mathcal K} \ip{x}{M_t} + \frac{1}{\eta_t}D_{\mathcal R}(x, y_{t-1})$
		\State $y_t = \arg \min\limits_{y \in \mathcal K} \ip{y}{g_t} + \frac{1}{\eta_t}D_{\mathcal R}(y, y_{t-1})$
	\EndIndent
\EndFor
\end{algorithmic}
\end{algorithm}
\vspace{-2mm}
Now, let us explain the salient differences between UniXGrad and MP as well as OMD using the particular choices of $M_t$, $g_t$ and the distance-generating function $\mathcal R$.

Optimistic Mirror Descent takes $g_t = \nabla f(x_t)$ and computes $M_t = \nabla f(x_{t-1})$, i.e., based on gradient information from previous iterates. This vector is available at the beginning of each iteration and the ``optimism'' arises in the case where $M_t \approx g_{t}$. When $M_t = \nabla f (y_{t-1})$ and $g_t = \nabla f(x_t)$, the template is known as the famous Mirror-Prox algorithm. One special case of Mirror-Prox is Extra-Gradient scheme \citep{korpelevich1976extragradient} where the projections are with respect to Euclidean norm, i.e. $\mathcal R (x) = 1 / 2 \norm{x}_2^2$, instead of general Bregman divergences.

MP has been well-studied, especially in the context of variational inequalities and convex-concave saddle point problems. It achieves fast convergence rate of $\mathcal O (1 / T)$ for this class of problems, however, in the context of smooth convex optimization, this is the standard slow rate \citep{nesterov2003introductory}. To date, MP is not known to enjoy the accelerated rate of $\mathcal O (1 / T^2)$ for smooth convex minimization.

We propose three modifications to this template, which are the precise choice of $g_t$ and $M_t$, the adaptive learning rate and the gradient weighting scheme.
\vspace{-2mm}
\paragraph{The notion of averaging:} In different interpretations of acceleration \citep{nesterov1983acceleration, tseng2008accelerated, orecchia2014coupling}, the notion of averaging is always central and we incorporate this notion via gradients taken at weighted average of iterates. Let us define the weight $\alpha_t = t$ and the following quantities

\vspace{-4mm}
\begin{align} \label{prelim:avg-iterates}
	\xbar{t} = \frac{\alpha_t x_t + \sum_{i=1}^{t-1} \alpha_i x_i }{\sum_{i=1}^{t} \alpha_i},  \quad\quad\quad  \ztilde{t} = \frac{ \alpha_t y_{t-1} + \sum_{i=1}^{t-1} \alpha_i x_i }{\sum_{i=1}^{t} \alpha_i }.
\end{align}
\vspace{-3mm}

Then, UniXGrad algorithm takes $g_t = \nabla f(\bar{x}_t)$ and $M_t = \nabla f(\tilde{z}_t)$, which provides a naive interpretation of averaging. Our choice of $g_t$ and $M_t$ coincide with that of the accelerated Extra-Gradient scheme of \citet{diakonikolas2017accelerated}. While their decision relies on implicit Euler discretization of an accelerated dynamics, we arrive at the same conclusion as a direct consequence of our convergence analysis.
\vspace{-2mm}
\paragraph{Adaptive learning rate:} A key ingredients of our algorithm is the choice of adaptive learning rate $\eta_t$. In light of \cite{rakhlin2013optimization}, we define our lag-one-behind learning rate as

\vspace{-5mm}
\begin{align} \label{prelim:learning-rate}
	\eta_t = \frac{2 D}{\sqrt{ 1 + \sum\limits_{i=1}^{t-1} \alpha_i^2 \norm{g_i - M_i}_\ast^2} },
\end{align}

where $D^2 = \sup_{x, y \in \mathcal K} D_{\mathcal R}(x,y)$ is the diameter of the compact set $\mathcal K$ with respect to Bregman divergences.
%We call it ``lag-one-behind'' as $g_t$ and $M_t$ are available together only after observing $x_t$, hence $\norm{g_t - M_t}^2_\ast$ is not computable until next iteration. 
Algorithm~\ref{alg:UniXGrad} summarizes our framework.

\vspace{-2mm}
\paragraph{Gradient weighting scheme:} We introduce the weights $\alpha_t$ in the sequence updates. One can interpret this as separating step size into learning rate and the scaling factors. It is necessary that $\alpha_t = \Theta(t)$ in order to achieve optimal rates, in fact we precisely choose $\alpha_t = t$. Also notice that they appear in the learning rate, compatible with the update rule.

\begin{algorithm}[H]
\caption{UniXGrad} \label{alg:UniXGrad}
\begin{algorithmic}[1]
\Input{ \# of iterations $T$ , $y_0 \in K$, diameter $D$, weight $\alpha_t = t$, learning rate $\bc{\eta_t}_{t \in [T]}$}
\For{$t = 1, ..., T$}
	\Indent
		\State $x_t = \arg \min\limits_{x \in \mathcal K} \alpha_t \ip{x}{M_t} + \frac{1}{\eta_t}D_{\mathcal R}(x, y_{t-1}) \quad \quad(M_t = \gradf{ \ztilde{t} })$
		\State $y_t = \arg \min\limits_{y \in \mathcal K} \alpha_t \ip{y}{g_t} + \frac{1}{\eta_t}D_{\mathcal R}(y, y_{t-1}) \quad \quad (g_t = \gradf{ \xbar{t} })$
	\EndIndent
\EndFor
\State \textbf{return} $\xbar{T}$
\end{algorithmic}
\end{algorithm}
\vspace{-2mm}
%Before proceeding with the analysis of the algorithm in smooth/non-smooth and deterministic/stochastic settings, we would like to provide two key technicalities that are required for proving our main results.
%
%\begin{lemma} \label{lem:technical1}
%	Let $\bc{a_i}_{i=1, ..., n}$ be a sequence of non negative numbers. Then, it holds that
%	
%	\[
%		\sqrt{ \sum_{i=1}^{n} a_i } \leq \sum_{i=1}^{n} \frac{ a_i }{\sum_{j=1}^{i} a_j } \leq 2 \sqrt{ \sum_{i=1}^{n} a_i }.
%	\]
%	
%\end{lemma}
%
%Please refer to \cite{levy2018accelegrad} for the proof of the Lemma~\ref{lem:technical1}. In addition, we will make use of the following inequality due to \cite{rakhlin2013omd}, which is slightly modified to agree with our weighted regret analysis.
%
%\begin{align} \label{def:gen-youngs-ineq}
%	\alpha_t \norm{g_t - M_t}_\ast \norm{x_t - y_t} = \inf_{\rho > 0} \bc{ \frac{\rho}{2} \alpha_t^2 \norm{g_t - M_t}_{\ast}^2 + \frac{1}{2 \rho} \norm{x_t - y_t}^2 }.
%\end{align}

In the remainder of this section, we will present our convergence theorems and provide proof sketches to emphasize the fundamental aspects and novelties. With the purpose of simplifying the analysis, we borrow classical tools in the online learning literature and perform the convergence analysis in the sense of bounding ``weighted regret''. Then, we use a simple yet essential conversion strategy which enables us to \emph{directly} translate our weighted regret bounds to convergence rates. Before we proceed, we will present  the conversion scheme from weighted regret to convergence rate, by deferring the proof to Appendix. In a concurrent work, \citep{cutkosky2019anytimeICML} proves a similar online-to-offline conversion bound.

\begin{lemma} \label{lem:regret-to-rate}
	Consider weighted average $\bar{x}_t$ as in Eq.~(\ref{prelim:avg-iterates}). Let $R_T(x_\ast) = \sum_{t=1}^{T} \alpha_t \ip{x_t - x_\ast}{g_t} $ denote the weighted regret after T iterations, $\alpha_t = t$ and $g_t = \nabla f(\bar{x}_t)$. Then,
	\[
		f(\bar{x}_T) - f(x_\ast) \leq \frac{2 R_T(x_\ast)}{T^2}.
	\]
	
\end{lemma}

\subsection{Non-smooth setting} \label{sec:nonsmooth-case}
\vspace{-2mm}
\paragraph{Deterministic setting:}First, we will focus on the convergence analysis in the case of non-smooth objective functions with deterministic/stochastic first-order oracles. We will follow the regret analysis as in \cite{rakhlin2013optimization} with essential adjustments that suit our weighted scheme and particular choice of adaptive learning rate.

\begin{remark}
	It is important to point out that we do not completely exploit the precise definitions of $g_t$ and $M_t$ in the presence of non-smooth objectives. As far as the regret analysis is concerned, it suffices that these quantities are functions of $\nabla f(\cdot)$ and that, as a corollary, their dual norm is upper bounded. However, in order to bridge the gap between weighted regret and the objective sub-optimality, i.e. $f(\bar{x}_T) - f(x_\ast)$, we require $g_t = \nabla f(\bar{x}_t)$.
\end{remark}

Now, we can exhibit our convergence bounds for the case of deterministic oracles.

\begin{theorem} \label{thm:rate-nonsmooth-deterministic}
	Consider the constrained optimization setting in Problem~(\ref{problem-def}), where $f: \mathcal K \rightarrow \mathbb R$ is a proper, convex and $G$-Lipschitz function defined over compact, convex set $\mathcal K$. Let $x^\ast \in \min_{x \in \mathcal K} f(x)$. Then, Algorithm~\ref{alg:UniXGrad} guarantees
	
	\vspace{-4mm}
	\begin{align}
		f( \bar{x}_T ) - \min_{x \in \mathcal K} f(x) \leq \frac{7 D \sqrt{ 1 + \sum_{t=1}^{T} \alpha_t^2 \norm{ g_t - M_t }_\ast^2 } - D}{T^2} \leq\frac{6 D}{ T^2} + \frac{14G D}{\sqrt{T}}.
	\end{align}
	\vspace{-2mm}
\end{theorem}

We establish the basis of our analysis through Lemma 1 and Corollary 2 of \cite{rakhlin2013optimization}. Then, we build upon this base by exploiting the structure of the adaptive learning rate, the weights $\alpha_t$ and the bound on gradient norms to give adaptive convergence bounds.

\vspace{-2mm}
\paragraph{Stochastic setting:}Now, we further consider the case of stochastic gradients. We assume that the first-order oracles are unbiased (see Eq.~\eqref{def:stochastic-gradient}). We want to emphasize that our stochastic setting is \emph{not} restricted to the notion of additive noise, i.e. gradients corrupted with zero-mean noise. It essentially includes any estimate that recovers the full gradient in expectation, e.g. estimating gradient using mini batches. Additionally, we propagate the bounded gradient norm assumption to the stochastic oracles, such that $\| \tilde \nabla f(x) \|_\ast \leq G$, $\forall x \in \mathcal K$.

\begin{theorem} \label{thm:rate-nonsmooth-stochastic}
	Consider the optimization setting in Problem~(\ref{problem-def}), where $f$ is non-smooth, convex and $G$-Lipschitz. Let $\{ x_t \}_{t=1,..,T}$ be a sequence generated by Algorithm~\ref{alg:UniXGrad} such that $g_t = \tilde{\nabla} f(\bar{x}_t)$ and $M_t = \tilde{\nabla} f(\tilde{z}_t)$. With $\alpha_t = t$ and learning rate as in Eq.~(\ref{prelim:learning-rate}), it holds that
	\[
		\mathbb E \left[ f(\bar{x}_T) \right] - \min\limits_{x \in \mathcal K} f(x) \leq \frac{6 D}{ T^2} + \frac{14 G D}{\sqrt{T}}.
	\]
	
\end{theorem}

The analysis in the stochastic setting is similar to deterministic setting. The difference is up to replacing $g_t \leftrightarrow \tilde{g}_t$ and $M_t \leftrightarrow \tilde{M}_t$. With the bound on stochastic gradients, the same rate is achieved.

%We decompose $g_t$ as $\tilde{g}_t + (g_t - \tilde{g}_t)$ and rewrite the weighted regret as
%
%\vspace{-4mm}
%\begin{align*}
%	R_T(x_\ast) &= \sum_{t=1}^{T} \alpha_t \ip{x_t - x_\ast}{\tilde{g}_t} + \sum_{t=1}^{T} \alpha_t \ip{x_t - x^\ast}{g_t - \tilde{g_t}}
%\end{align*}
%
%Due to unbiasedness of the gradient estimates, expected value of $\alpha_t \ip{x_t - x^\ast}{g_t - \tilde{g_t}}$, conditioned on the average iterate $\bar{x}_t$ evaluates to 0. We will only need to bound the first summation whose analysis is identical to its deterministic counterpart up to replacing $g_t$ with $\tilde{g}_t$, and $M_t$ with $\tilde{M}_t$. Therefore, we safely obtain the same worst-case bound with the deterministic oracle scheme without suffering neither $\log{T}$ nor $\sigma$ dependencies on the convergence rate.

%\[
%	f(\bar{x}_T) - \min\limits_{x \in \mathcal K} f(x) \leq \underbrace{ \frac{1}{A_T} \sum_{t=1}^{T} \alpha_t \ip{x_t - x^\ast}{\tilde{g_t}} }_\textrm{(A)} + \underbrace{ \frac{1}{A_T} \sum_{t=1}^{T} \alpha_t \ip{x_t - x^\ast}{g_t - \tilde{g_t}} }_\textrm{(B)}.
%\]
%
%When we take the conditional expectation given the average iterate $\bar{x}_t$, observe that the term (B) evaluates to zero. For term (A), the proof is almost exactly the same as before. We use Jensen's inequality to take expectation into the square root and use the boundedness assumption on the gradients to obtain the bound.

\subsection{Smooth setting} \label{sec:smooth-case}

\paragraph{Deterministic setting:}In terms of theoretical contributions and novelty, the case of $L$-smooth objective is of greater interest.
%In the non-smooth analysis, we ignore a negative summation term. When coupled with smoothness of the objective and a particular characterization of the growth of learning rate $\eta_t$, we will obtain the accelerated rates both in deterministic and stochastic settings.
We will first start with the deterministic oracle scheme and then introduce the convergence theorem for the noisy setting.

\begin{theorem} \label{thm:rate-smooth-deterministic}
	Consider the constrained optimization setting in Problem~(\ref{problem-def}), where $f: \mathcal K \rightarrow \mathbb R$ is a proper, convex and $L$-smooth function defined over compact, convex set $\mathcal K$. Let $x^\ast \in \min_{x \in \mathcal K} f(x)$. Then, Algorithm~\ref{alg:UniXGrad} ensures the following
	\vspace{-1mm}
	\begin{align}
		f(\bar{x}_T) - \min\limits_{x \in \mathcal K} f(x) \leq \frac{20 \sqrt{7} D^2 L}{T^2}.
	\end{align}
\end{theorem}

\begin{remark}
	In the non-smooth setting, we assume that gradients have bounded norms. Our algorithm does \textbf{not} need to know this information, but it is necessary for the analysis in that case. However, when the function is smooth, neither the algorithm nor the analysis requires bounded gradients.
\end{remark}

\begin{proof}[Proof Sketch (Theorem~\ref{thm:rate-smooth-deterministic})] We follow the proof of Theorem~\ref{thm:rate-nonsmooth-deterministic} until the point where we obtain

\vspace{-4mm}
\begin{align*}
	\sum_{t=1}^{T} \alpha_t \ip{x_t - x_\ast}{g_t} &\leq \frac{1}{2} \sum_{t=1}^{T} \eta_{t+1} \alpha_t^2 \norm{ g_t - M_t }_\ast^2 - \frac{1}{2} \sum_{t=1}^{T} \frac{1}{\eta_{t+1}} \norm{x_t - y_{t-1}}^2 + D^2 \br{ \frac{3}{\eta_{T+1}} + \frac{1}{\eta_1} } .
\end{align*}

By smoothness of the objective function, we have $\norm{g_t - M_t}_\ast \leq L \norm{\bar{x}_t - \tilde{z}_t}$, which implies $- \frac{1}{\eta_{t+1}} \norm{x_t - y_{t-1}}^2 \leq - \frac{\alpha_t^2}{4 L^2 \eta_{t+1}} \norm{g_t - M_t}_\ast^2$. Hence,

\vspace{-3mm}
\begin{align*}
	\leq \frac{1}{2} \sum_{t=1}^{T} \br{ \eta_{t+1} - \frac{1}{4L^2 \eta_{t+1}} } \alpha_t^2 \norm{g_t - M_t}_\ast^2 + D^2 \br{ \frac{3}{\eta_{T+1}} + \frac{1}{\eta_1} } .
\end{align*}

Now we will introduce a time variable to \emph{characterize} the growth of the learning rate. Define $\tau^\ast = \max \bc{ t \in \bc{ 1, ..., T } \, : \, \frac{1}{\eta_{t+1}^2} \leq 7 L^2}$ such that $\forall t > \tau^\ast$, $\eta_{t+1} - \frac{1}{4 L^2 \eta_{t+1}} \leq -\frac{3}{4} \eta_{t+1} $. Then,

\vspace{-3mm}
\begin{align*}
	&\leq  \underbrace{ D \sum_{t=1}^{\tau^\ast} \frac{ \alpha_t^2 \norm{g_t - M_t}_\ast^2 }{ \sqrt{ 1 + \sum_{i=1}^{t} \alpha_i^2 \norm{g_i - M_i}_\ast^2 } } + \frac{ D}{2} }_\textrm{(A)} \\
	&+ \underbrace{ \frac{3 D}{2} \br{  \sqrt{ 1 + \sum_{t=1}^{T} \alpha_t^2 \norm{ g_t - M_t } _\ast^2 } - \sum_{t=\tau^\ast + 1}^{T} \frac{\alpha_t^2 \norm{g_t - M_t}_\ast^2}{\sqrt{ 1 + \sum_{i=1}^{t} \alpha_i^2 \norm{g_i - M_i}_\ast^2 }} } }_\textrm{(B)},
\end{align*}

where we wrote $\eta_{t+1}$ in open form and used the definition of $\tau^\ast$. To complete the proof, we will need the following lemma.

\begin{lemma} \label{lem:technical1}
	Let $\bc{a_i}_{i=1, ..., n}$ be a sequence of non negative numbers. Then, it holds that
	
	\[
		\sqrt{ \sum_{i=1}^{n} a_i } \leq \sum_{i=1}^{n} \frac{ a_i }{\sum_{j=1}^{i} a_j } \leq 2 \sqrt{ \sum_{i=1}^{n} a_i }.
	\]
	
\end{lemma}

Please refer to \citep{mcmahan2010adaptive, levy2018online} for the proof. We jointly use Lemma~\ref{lem:technical1} and the bound on $\eta_{\tau* + 1}$ to upper bound terms (A) and (B) with $4 \sqrt{7} D^2 L$ and $6 \sqrt{7} D^2 L$, respectively. Lemma~\ref{lem:regret-to-rate} immediately establishes the convergence bound.
\end{proof}

\paragraph{Stochastic setting:}Next, we will present our results for the stochastic extension. In addition to unbiasedness and boundedness, we will introduce another classical assumption: bounded variance,

\vspace{-3mm}
\begin{align} \label{def:bounded-variance}
	E [\|\nabla f(x) - \tilde{\nabla}f(x)\|_\ast^2 \vert x] \leq \sigma^2, \quad \forall x \in \mathcal K.
\end{align}

The analysis proceeds along similar lines as its deterministic counterpart. However, we execute the analysis using auxiliary terms and attain the optimal accelerated rate without the log factors. 

\begin{theorem} \label{thm:rate-smooth-stochastic}
	Consider the optimization setting in Problem~\eqref{problem-def}, where $f$ is $L$-smooth and convex. Let $\{ x_t \}_{t=1,..,T}$ be a sequence generated by Algorithm~\ref{alg:UniXGrad} such that $g_t = \tilde{\nabla} f(\bar{x}_t)$ and $M_t = \tilde{\nabla} f(\tilde{z}_t)$. With $\alpha_t = t$ and learning rate as in \eqref{prelim:learning-rate}, it holds that
	\vspace{-1mm}
	\[
		\mathbb E \left[ f(\bar{x}_T) \right] - \min\limits_{x \in \mathcal K} f(x) \leq \frac{224 \sqrt{14} D^2 L }{T^2} + \frac{14\sqrt{2} \sigma D }{\sqrt{T}}.
	\]
	
\end{theorem}

\begin{proof}[Proof Sketch (Theorem~\ref{thm:rate-smooth-stochastic})] 

	We start in the same spirit as the stochastic, non-smooth setting,
	\vspace{-1mm}
	\[
	\sum_{t=1}^{T} \alpha_t \ip{x_t - x_\ast}{g_t} \leq \underbrace{ \sum_{t=1}^{T} \alpha_t \ip{x_t - x^\ast}{\tilde{g_t}} }_\textrm{(A)} + \underbrace{ \sum_{t=1}^{T} \alpha_t \ip{x_t - x^\ast}{g_t - \tilde{g_t}} }_\textrm{(B)}.
	\]
	
	Recall that term (B) is zero in expectation given $\bar{x}_t$. Then, we follow the proof steps of Theorem~\ref{thm:rate-nonsmooth-deterministic},
	\vspace{-1mm}
	\begin{align} \label{prf:middle-step}
		\sum_{t=1}^{T} \alpha_t \ip{x_t - x_\ast}{g_t} \leq \frac{7 D}{2} \sqrt{ 1 + \sum_{t=1}^{T} \alpha_t^2 \| \tilde{g}_t - \tilde{M}_t \|_\ast^2 } - \frac{1}{2} \sum_{t=1}^{T} \frac{1}{\eta_{t+1}} \norm{x_t - y_{t-1}}^2.
	\end{align}
	
	We will obtain $\norm{g_t - M_t}_\ast^2$ from $\norm{x_t - y_{t-1}}^2$ due to smoothness and the challenge is to handle $\| \tilde{g}_t - \tilde{M}_t \|_\ast^2$ and $\norm{g_t - M_t}_\ast^2$ together. So let's denote, $ B_t^2: = \min \{\|g_t-M_t\|_\ast^2, \| \tg_t -\tM_t\|_\ast^2\} $. Using this definition, we could declare an auxiliary learning rate which we will \emph{only} use for the analysis,
	\vspace{-0.5mm}
	\begin{align} \label{def:aux-learning-rate}
		\teta_t = \frac{2 D}{\sqrt{ 1 + \sum\limits_{i=1}^{t-1} \alpha_i^2 B_i^2}} .
	\end{align}
	
	Clearly, for any $t\in[T]$ we have $-\frac{1}{\eta_{t+1}} \norm{g_t - M_t}_\ast^2 \leq  -\frac{1}{\teta_{t+1}}B_t^2$.
        Also, we can write,
	
        \begin{align}\label{eq:Ineq33}
        		\|\tg_t -\tM_t\|_\ast^2 \leq 2\|g_t-M_t\|_\ast^2  +2\|\xi_t\|_\ast^2,
        \end{align}
        and,
        \begin{align*}
        \|\tg_t -\tM_t\|_\ast^2 \leq 2B_t^2 + 2\|\xi_t\|_\ast^2.
        \end{align*}
        
        Therefore, we could rewrite Eq.~\eqref{prf:middle-step} as,
        \vspace{-1mm}
        \[
        		\sum_{t=1}^{T} \alpha_t \ip{x_t - x_\ast}{g_t} \leq  \underbrace{ \frac{7}{2} \sum_{t=1}^{T} \left( \teta_{t+1} - \frac{1}{28 L^2 \teta_{t+1}} \right) \alpha_t^2 B_t^2 + \frac{7  D}{2} }_\textrm{(A)} + \underbrace{ \frac{7 D}{\sqrt{2}} \sqrt{ \sum_{t=1}^{T} \alpha_t^2 \norm{\xi_t}_\ast^2 }  }_\textrm{(B)}.
        \]
        Using Lemma~\ref{lem:technical1} and defining a time variable $\tau_\ast$ in the sense of Theorem~\ref{thm:rate-smooth-deterministic} (with correct constants), term (A) is upper bounded by $112 \sqrt{14} D^2 L$. By taking expectation conditioned on $\bar{x}_t$ and using Jensen's inequality, we could upper bound term (B) as $14 \sigma D T^{3/2} / \sqrt{2} $, which leads us to the optimal rate of ${224 \sqrt{14} D^2 L}/{T^2} + {14 \sqrt{2} \sigma D }/{\sqrt{ T}}$ through Lemma~\ref{lem:regret-to-rate}. 	
\end{proof}

%%%%%%%%%%%%%%%%%%%%%%%%%%%%%%%%%%%%%%%%%%%%%%%%%%%%%%%%%%%%%%%%%%%%%%%%%%%%%%%%%%%%%%%%%%%%%%
%%%%%%%%%%%%%%%%%%%%%%%%%%%%%%%%%%%%%%%%%%%%%%%%%%%%%%%%%%%%%%%%%%%%%%%%%%%%%%%%%%%%%%%%%%%%%%
%!TEX root = main.tex
%%%%%%%%%%%%%%%%%%%%%%%%%%%%%%%%%%%%%%%%%%%%%%%%%%%%%%%%%%%%%%%%%%%%%%%%%%%%%%%%%%%%%%%%%%%%%%
%%%%%%%%%%%%%%%%%%%%%%%%%%%%%%%%%%%%%%%%%%%%%%%%%%%%%%%%%%%%%%%%%%%%%%%%%%%%%%%%%%%%%%%%%%%%%%

\section{Experiments} \label{sec:experiments}

We compare performance of our algorithm for two different tasks against adaptive methods of various characteristics, such as AdaGrad, AMSGrad and AcceleGrad, along with a recent non-adaptive method AXGD. We consider a synthetic setting where we analyze the convergence behavior, as well as a SVM classification task on some LIBSVM dataset. In all the setups, we tuned the hyper-parameters of each algorithm by grid search. In order to compare the adaptive methods on equal grounds, AdaGrad is implemented with a scalar step size based on the template given by \citet{levy2017online}. We implement AMSGrad exactly as it is described by \citet{j.2018on}.

\subsection{Convergence behavior}

We take the least squares problem with $L_2$-norm ball constraint for this setting, i.e., $\min_{\norm{x}_2 < r} \frac{1}{2n} \norm{ Ax - b }_2^2$, where $A \in \mathbb R^{n \times d}$, $A \sim \mathcal N (0, \sigma^2 I)$ and $b = A x^\natural + \epsilon$ such that $\epsilon$ is a random vector $ \sim \mathcal N (0, 10^{-3})$. We pick $n = 500$ and $d = 100$. For the rest of this section, we refer to the solution of \emph{constrained} problem as $x_\ast$.

%\begin{figure}[ht]
%\begin{subfigure}{.49\textwidth}
%  \centering
%  % include first image
%  \includegraphics[width=.8\linewidth]{figs/ConvergenceRate_AvgIterate_All_experiment-type=LeastSquares-L2-LargeR-Deterministic_d=100.pdf}  
%  \caption{Average Iterate}
%  \label{fig:sub-lrg-det-avg}
%\end{subfigure}
%\begin{subfigure}{.49\textwidth}
%  \centering
%  % include second image
%  \includegraphics[width=.8\linewidth]{figs/ConvergenceRate_LastIterate_All_experiment-type=LeastSquares-L2-LargeR-Deterministic_d=100.pdf}  
%  \caption{Last Iterate}
%  \label{fig:sub-lrg-det-last}
%\end{subfigure}
%\caption{Convergence rates in the \textbf{deterministic} oracle setting when the minimizer of $f$ is \textbf{inside} the constraint set}
%\label{fig:synthetic-deterministic1}
%\end{figure}

\begin{figure}[H]
\centering
\begin{subfigure}{.49\textwidth}
  \centering
  % include first image
  \includegraphics[width=.8\linewidth]{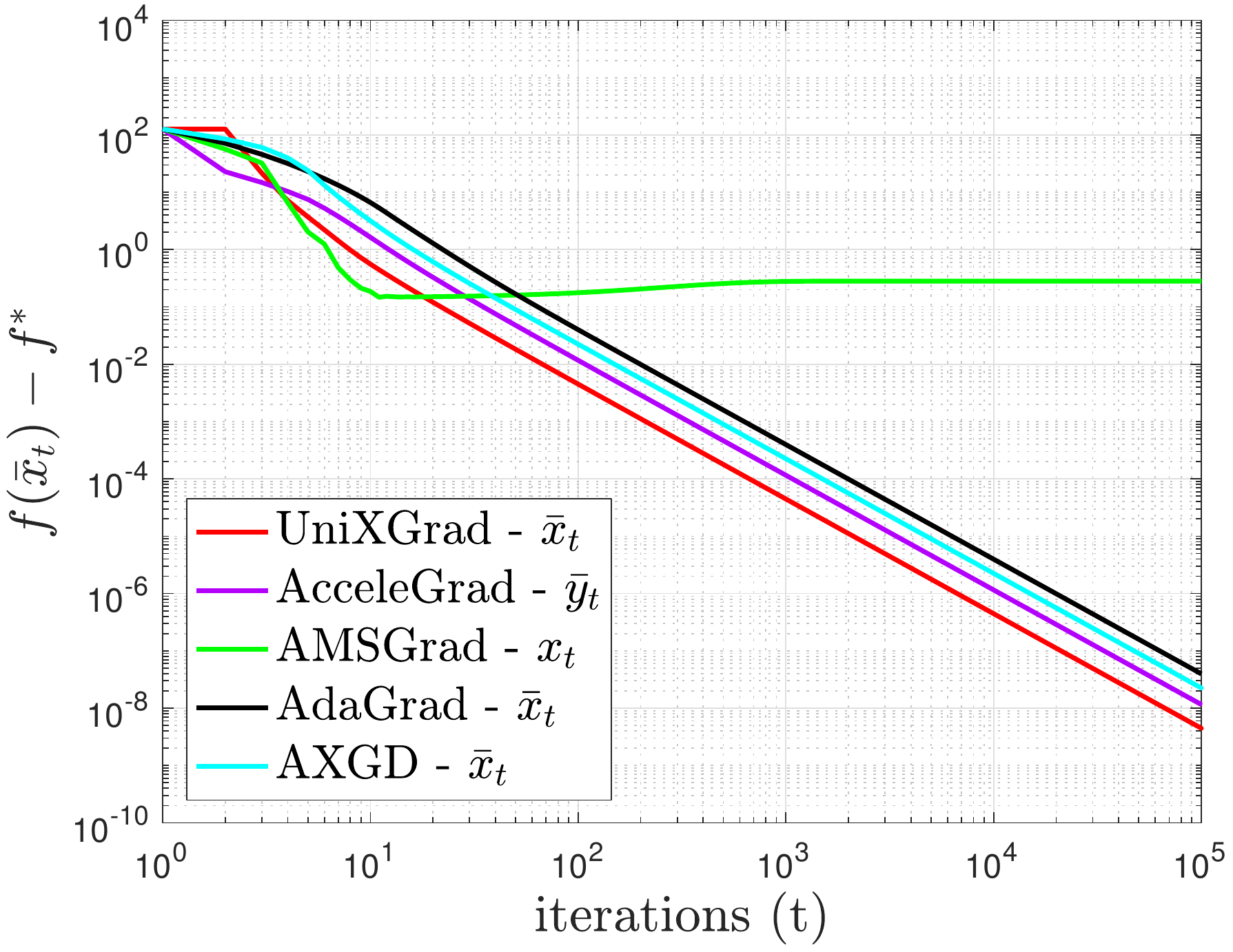}  
  \caption{Average Iterate}
  \label{fig:sub-sml-det-avg}
\end{subfigure}
\begin{subfigure}{.49\textwidth}
  \centering
  % include second image
  \includegraphics[width=.8\linewidth]{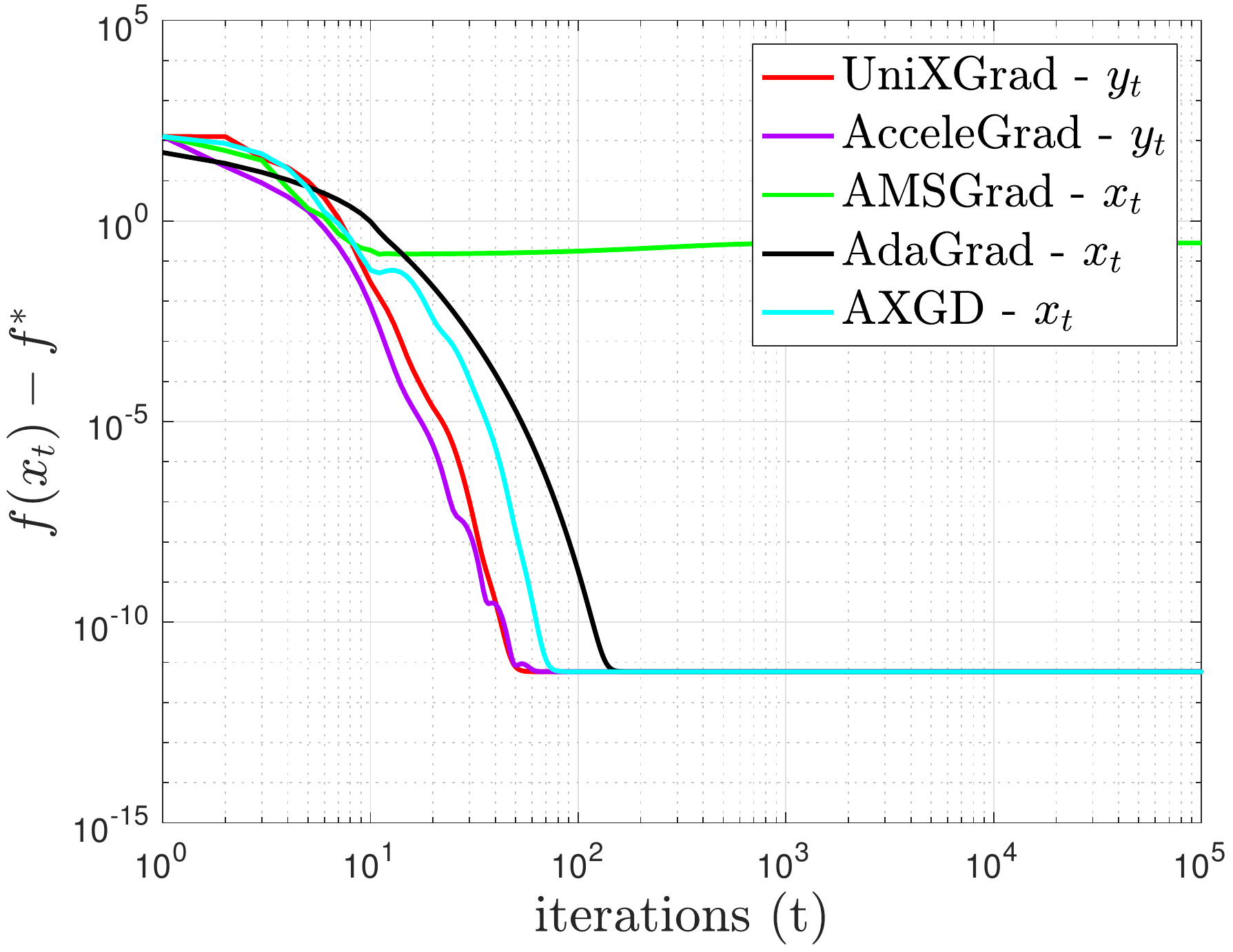}  
  \caption{Last Iterate}
  \label{fig:sub-sml-det-last}
\end{subfigure}
\caption{Convergence rates in the \textbf{deterministic} oracle setting when $x_\ast \in \text{Boundary}(\mathcal K)$}
\label{fig:synthetic-deterministic2}
\end{figure}

\begin{figure}[H]
\centering
\begin{subfigure}{.49\textwidth}
  \centering
  % include first image
  \includegraphics[width=.8\linewidth]{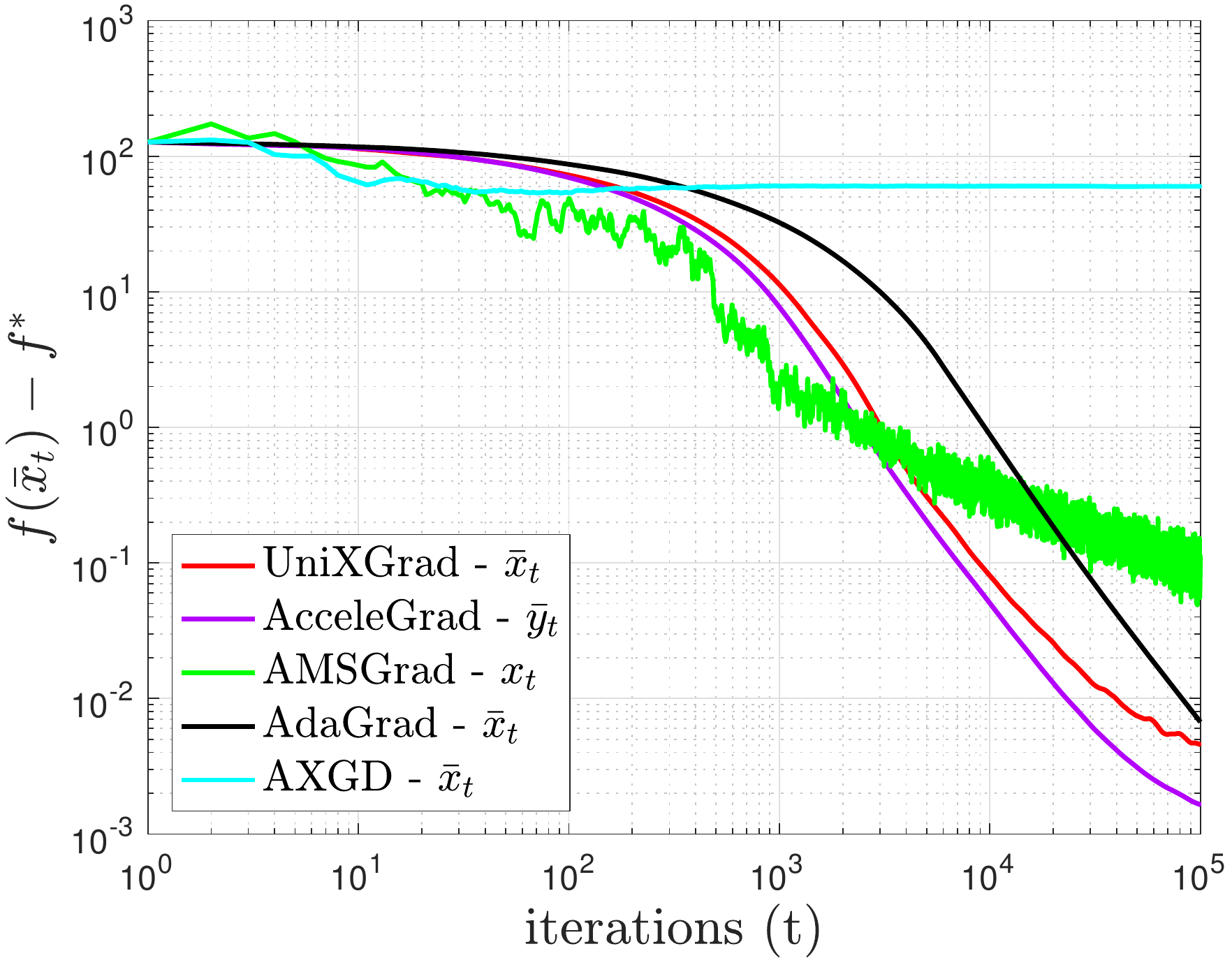}  
  \caption{Average Iterate}
  \label{fig:sub-sml-stc-avg}
\end{subfigure}
\begin{subfigure}{.49\textwidth}
  \centering
  % include second image
  \includegraphics[width=.8\linewidth]{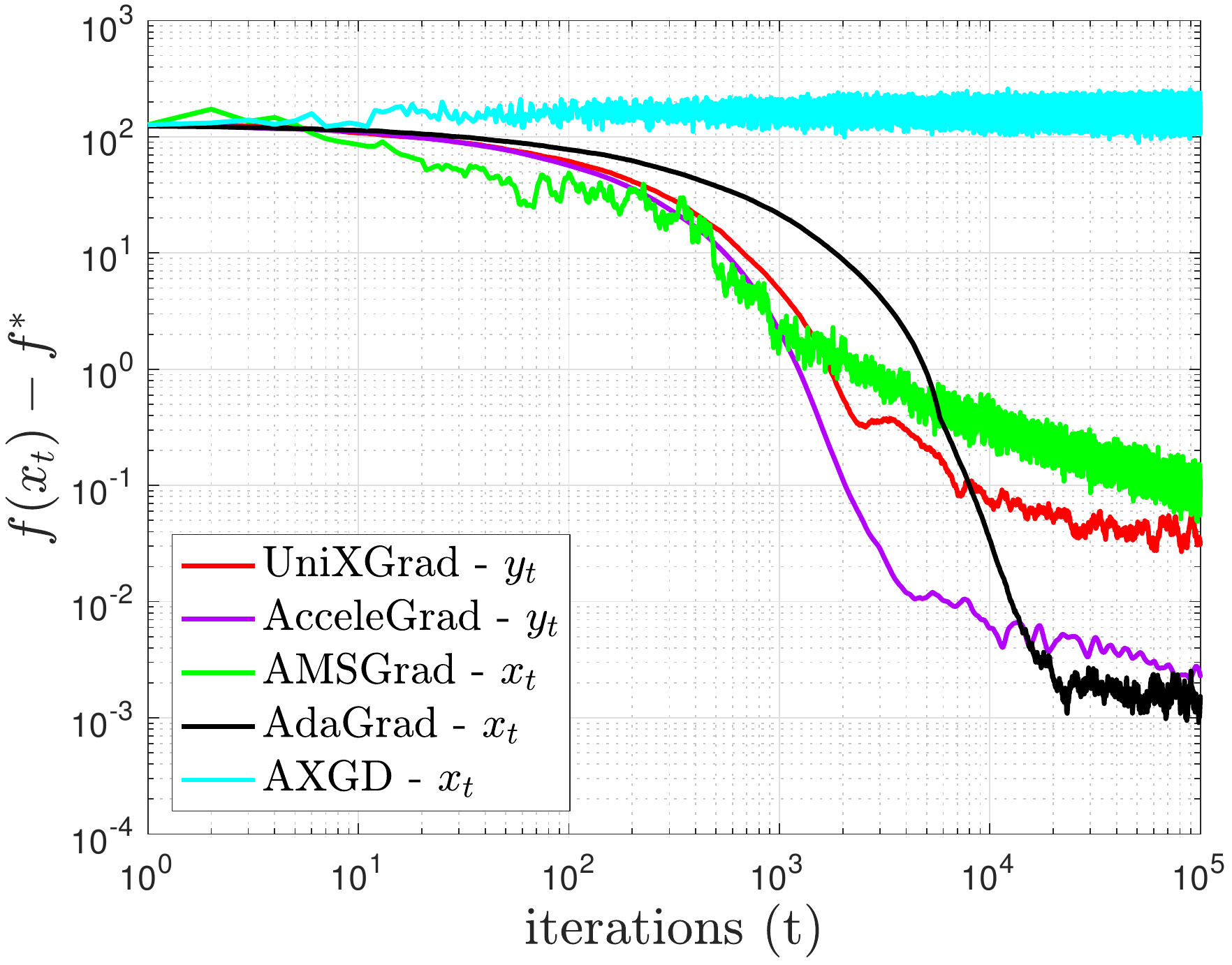}  
  \caption{Last Iterate}
  \label{fig:sub-sml-stc-last}
\end{subfigure}
\caption{Convergence rates in the \textbf{stochastic} oracle setting when $x_\ast \in \text{Boundary}(\mathcal K)$}
\label{fig:synthetic-stochastic2}
\end{figure}

\vspace{-3mm}
In Figure~\ref{fig:synthetic-deterministic2} and \ref{fig:synthetic-stochastic2}, we present the convergence rates under deterministic and stochastic oracles, and we pick a problem in which the solution is on the boundary of the constraint set, i.e., $x_\ast \in \text{Boundary}(\mathcal K)$. In this setting, our algorithm shows matching performance in comparison with other methods. AXGD has convergence issues in the stochastic setting, as it only handles additive noise and their step size routine does not seem to be compatible with stochastic gradients. Another key observation is that AMSGrad suffers a decrease in its performance when the solution is on the boundary of the set.

%\begin{figure}[ht]
%\begin{subfigure}{.49\textwidth}
%  \centering
%  % include first image
%  \includegraphics[width=.8\linewidth]{figs/ConvergenceRate_AvgIterate_All_experiment-type=LeastSquares-L2-LargeR-Stochastic_d=100.pdf}  
%  \caption{Average Iterate}
%  \label{fig:sub-lrg-stc-avg}
%\end{subfigure}
%\begin{subfigure}{.49\textwidth}
%  \centering
%  % include second image
%  \includegraphics[width=.8\linewidth]{figs/ConvergenceRate_LastIterate_All_experiment-type=LeastSquares-L2-LargeR-Stochastic_d=100.pdf}  
%  \caption{Last Iterate}
%  \label{fig:sub-lrg-stc-last}
%\end{subfigure}
%\caption{Convergence rates in the \textbf{stochastic} oracle setting when the minimizer of $f$ is \textbf{inside} the constraint set}
%\label{fig:synthetic-stochastic1}
%\end{figure}

%Figure~\ref{fig:synthetic-stochastic1} and \ref{fig:synthetic-stochastic2} present the convergence performance in stochastic setting, where we observe a slightly different scenario. When $x_{\text{opt}}$ is inside the constrained set, AdaGrad and AMSGrad outperform the universal, adaptive methods Accelegrad and UniOMD. Step size routine for AXGD, which is the non-universal and non-adaptive counterpart of our framework, is not compatible with the minibatch setting. When the deterministic gradients are corrupted with noise, AXGD maintains its performance, however, it is not robust to stochasticity due to mini batch randomness. In the case where $x_\ast \not= x_{\text{opt}}$, our method achieves comparable performance.

\subsection{SVM classification}

In this section, we will tackle SVM classification problem on ``breast-cancer'' data set taken from LIBSVM. We try to minimize squared Hinge loss with $L_2$ norm regularization. We split the data set as training and test sets with 80/20 ratio. The models are trained using random mini batches of size 5. Figure~\ref{fig:classification} demonstrates convergence rates and test accuracies of the methods. They represent the average performance of 5 runs, with random initializations. For UniXGrad, AcceleGrad and AXGD, we consider the performance with respect to the average iterate $\bar{x}_t$ as it shows a more stable behavior, whereas AdaGrad and AMSGrad are evaluated based on their last iterates. AXGD, which has poor convergence behavior in stochastic setting due to its step size rule, shows the worst performance both in terms of convergence and generalization. UniXGrad, AcceleGrad, AdaGrad and AMSGrad achieve comparable generalization performances to each other. AMSGrad achieves a slightly better performance as it has diagonal preconditioner which translates to per-coordinate learning rate. It could possibly adapt to the geometry of the optimization landscape better. 

\begin{figure}[H]
\centering
\begin{subfigure}{.49\textwidth}
  \centering
  % include first image
  \includegraphics[width=.9\linewidth]{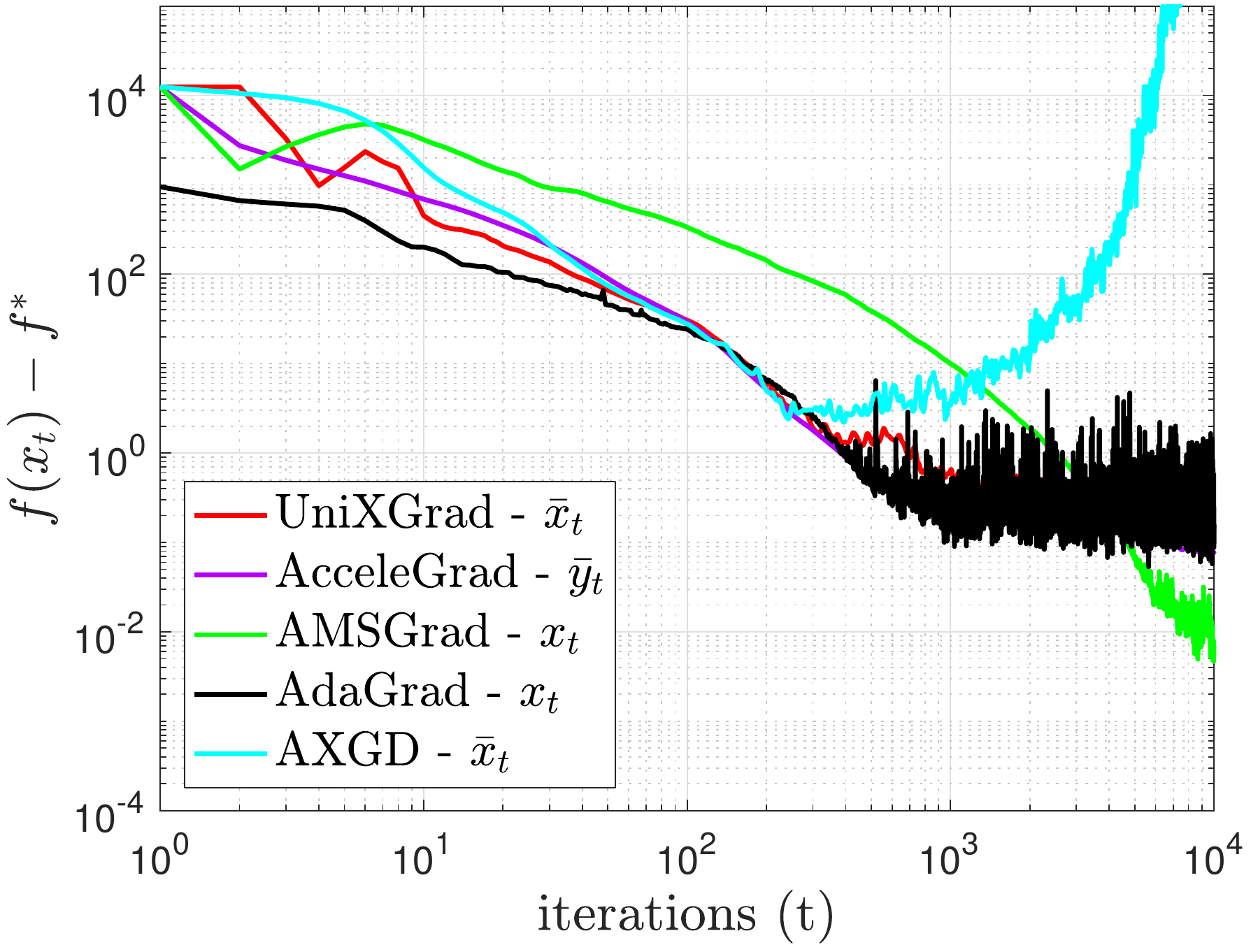}  
  \caption{Convergence rates with respect to training data}
  \label{fig:sub-brst-train}
\end{subfigure}
\begin{subfigure}{.49\textwidth}
  \centering
  % include second image
  \includegraphics[width=.9\linewidth]{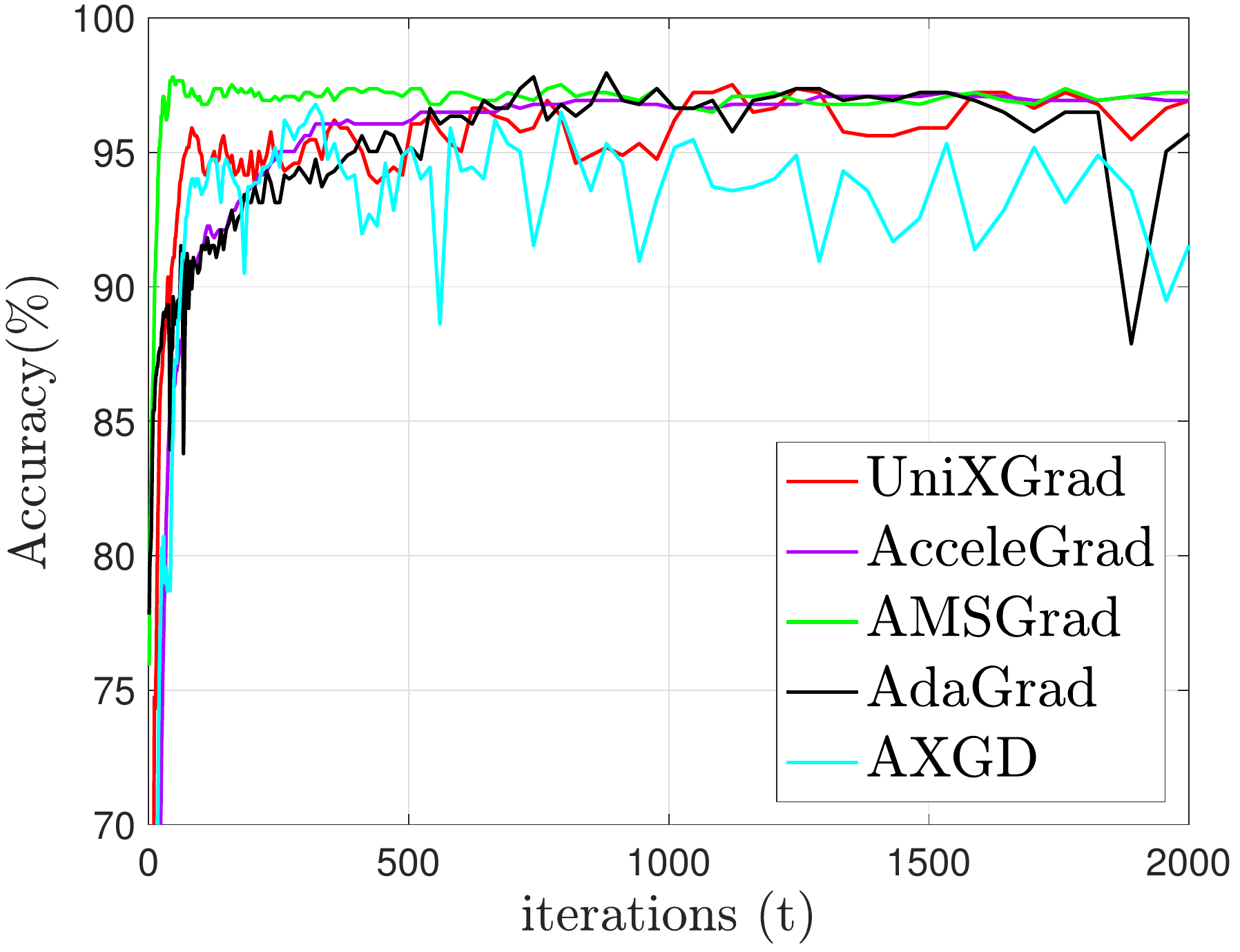}  
  \caption{Test Accuracy}
  \label{fig:sub-brst-test}
\end{subfigure}
\caption{Convergence behavior with respect to training data and resulting test accuracies for binary classification task on breast-cancer dataset from LIBSVM~\cite{libsvm}}
\label{fig:classification}
\end{figure}

%%%%%%%%%%%%%%%%%%%%%%%%%%%%%%%%%%%%%%%%%%%%%%%%%%%%%%%%%%%%%%%%%%%%%%%%%%%%%%%%%%%%%%%%%%%%%%
%%%%%%%%%%%%%%%%%%%%%%%%%%%%%%%%%%%%%%%%%%%%%%%%%%%%%%%%%%%%%%%%%%%%%%%%%%%%%%%%%%%%%%%%%%%%%%
%!TEX root = main.tex
%%%%%%%%%%%%%%%%%%%%%%%%%%%%%%%%%%%%%%%%%%%%%%%%%%%%%%%%%%%%%%%%%%%%%%%%%%%%%%%%%%%%%%%%%%%%%%
%%%%%%%%%%%%%%%%%%%%%%%%%%%%%%%%%%%%%%%%%%%%%%%%%%%%%%%%%%%%%%%%%%%%%%%%%%%%%%%%%%%%%%%%%%%%%%

\section{Discussion and Future Work} \label{sec:conclusion}
	In this paper we presented an adaptive and universal framework that achieves the optimal convergence rates in constrained convex optimization setting. To our knowledge, this is the first method that achieves $ \mathcal{O}\left(G D /\sqrt{T}\right)$ and $ \mathcal{O} \left(D^2 L/T^{2} + \sigma D /\sqrt{T}\right)$ rates in the constrained setting, without log dependencies. Without any prior information, our algorithm adapts to smoothness of the objective function as well as the variance of the possibly noisy gradients. 
	
	One would interpret that our guarantees are extensions of \citep{levy2018online} to the constrained setting, through a completely different algorithm and a simpler, classical analysis. Our study of their algorithm and proof strategies concludes that:
	\begin{itemize}
		\item It does not seem possible to remove $\log T$ dependency in non-smooth setting for their algorithm, due to their Lemma A.3
		\item Extending their algorithm to constrained setting (via projecting $y$ sequence) is not trivial, as the analysis requires $y$ sequence to be unbounded (refer to their Appendix A, Eq. (16)).
	\end{itemize}
	
	As a follow up to our work, we would like to investigate three main extensions: 
	
	\begin{itemize}
		\item Proximal version of our algorithm that could handle composite problems with nonsmooth terms, including indicator functions, in a unified manner. It seems like a rather simple extension as the main difference would be replacing optimality condition for constrained updates with that of proximal operator.
		\item Extending scalar adaptive learning rate to per-coordinate matrix-like preconditioner. This direction of research would help us create a robust algorithm that is applicable to non-convex problems, such as training deep neural networks.
		\item Adaptation to strong convexity along with smoothness and noise variance, simultaneously. A first step towards tackling this open problem is proving an improved rate of $O(1 / T^2 + \sigma / T)$ for smooth and strongly convex problems, with stochastic gradients. 
	\end{itemize}
	
\section*{Acknowledgment}
	AK and VC are supported by the European Research Council (ERC) under the European Union's Horizon 2020 research and innovation programme (grant agreement no 725594 - time-data) and the Swiss National Science Foundation (SNSF) under grant number 200021\_178865 / 1.

%%%%%%%%%%%%%%%%%%%%%%%%%%%%%%%%%%%%%%%%%%%%%%%%%%%%%%%%%%%%%%%%%%%%%%%%%%%%%%%%%%%%%%%%%%%%%%
%%%%%%%%%%%%%%%%%%%%%%%%%%%%%%%%%%%%%%%%%%%%%%%%%%%%%%%%%%%%%%%%%%%%%%%%%%%%%%%%%%%%%%%%%%%%%%
\bibliography{refs}

\begin{thebibliography}{22}
\providecommand{\natexlab}[1]{#1}
\providecommand{\url}[1]{\texttt{#1}}
\expandafter\ifx\csname urlstyle\endcsname\relax
  \providecommand{\doi}[1]{doi: #1}\else
  \providecommand{\doi}{doi: \begingroup \urlstyle{rm}\Url}\fi

\bibitem[{Allen Zhu} and Orecchia(2014)]{orecchia2014coupling}
Z.~{Allen Zhu} and L.~Orecchia.
\newblock A novel, simple interpretation of nesterov's accelerated method as a
  combination of gradient and mirror descent.
\newblock \emph{CoRR}, abs/1407.1537, 2014.

\bibitem[Bach and Levy(2019)]{bach2019universal}
F.~Bach and K.~Y. Levy.
\newblock A universal algorithm for variational inequalities adaptive to
  smoothness and noise.
\newblock \emph{arXiv preprint arXiv:1902.01637}, 2019.

\bibitem[Chang and Lin(2011)]{libsvm}
C.-C. Chang and C.-J. Lin.
\newblock {LIBSVM}: A library for support vector machines.
\newblock \emph{ACM Transactions on Intelligent Systems and Technology},
  2:\penalty0 27:1--27:27, 2011.

\bibitem[Cohen et~al.(2018)Cohen, Diakonikolas, and
  Orecchia]{cohen2018acceleration}
M.~B. Cohen, J.~Diakonikolas, and L.~Orecchia.
\newblock On acceleration with noise-corrupted gradients.
\newblock \emph{arXiv preprint arXiv:1805.12591}, 2018.

\bibitem[Cutkosky(2019)]{cutkosky2019anytimeICML}
A.~Cutkosky.
\newblock Anytime online-to-batch conversions, optimism, and acceleration.
\newblock \emph{the International Conference on Machine Learning (ICML)}, June
  2019.

\bibitem[Deng et~al.(2018)Deng, Cheng, and Lan]{deng2018optimal}
Q.~Deng, Y.~Cheng, and G.~Lan.
\newblock Optimal adaptive and accelerated stochastic gradient descent.
\newblock \emph{arXiv preprint arXiv:1810.00553}, 2018.

\bibitem[Diakonikolas and Orecchia(2017)]{diakonikolas2017accelerated}
J.~Diakonikolas and L.~Orecchia.
\newblock Accelerated extra-gradient descent: A novel accelerated first-order
  method.
\newblock \emph{arXiv preprint arXiv:1706.04680}, 2017.

\bibitem[Duchi et~al.(2011)Duchi, Hazan, and Singer]{duchi2011adaptive}
J.~Duchi, E.~Hazan, and Y.~Singer.
\newblock Adaptive subgradient methods for online learning and stochastic
  optimization.
\newblock \emph{Journal of Machine Learning Research}, 12\penalty0
  (Jul):\penalty0 2121--2159, 2011.

\bibitem[Hu et~al.(2009)Hu, Pan, and Kwok]{hu2009accelerated}
C.~Hu, W.~Pan, and J.~T. Kwok.
\newblock Accelerated gradient methods for stochastic optimization and online
  learning.
\newblock In \emph{Advances in Neural Information Processing Systems}, pages
  781--789, 2009.

\bibitem[Kingma and Ba(2014)]{kingma2014adam}
D.~Kingma and J.~Ba.
\newblock Adam: A method for stochastic optimization.
\newblock \emph{arXiv preprint arXiv:1412.6980}, 2014.

\bibitem[Korpelevich(1976)]{korpelevich1976extragradient}
G.~M. Korpelevich.
\newblock The extragradient method for finding saddle points and other
  problems.
\newblock \emph{Matecon}, 12:\penalty0 747--756, 1976.

\bibitem[Lan(2012)]{lan2012optimal}
G.~Lan.
\newblock An optimal method for stochastic composite optimization.
\newblock \emph{Mathematical Programming}, 133\penalty0 (1-2):\penalty0
  365--397, 2012.

\bibitem[Levy(2017)]{levy2017online}
K.~Levy.
\newblock Online to offline conversions, universality and adaptive minibatch
  sizes.
\newblock In \emph{Advances in Neural Information Processing Systems}, pages
  1612--1621, 2017.

\bibitem[Levy et~al.(2018)Levy, Yurtsever, and Cevher]{levy2018online}
K.~Y. Levy, A.~Yurtsever, and V.~Cevher.
\newblock Online adaptive methods, universality and acceleration.
\newblock In \emph{Neural and Information Processing Systems (NeurIPS)},
  December 2018.

\bibitem[McMahan and Streeter(2010)]{mcmahan2010adaptive}
H.~B. McMahan and M.~Streeter.
\newblock Adaptive bound optimization for online convex optimization.
\newblock \emph{COLT 2010}, page 244, 2010.

\bibitem[Nemirovski(2004)]{nemirovski2004prox}
A.~Nemirovski.
\newblock Prox-method with rate of convergence o(1/t) for variational
  inequalities with lipschitz continuous monotone operators and smooth
  convex-concave saddle point problems.
\newblock \emph{{SIAM} Journal on Optimization}, 15\penalty0 (1):\penalty0
  229--251, 2004.

\bibitem[Nesterov(1983)]{nesterov1983acceleration}
Y.~Nesterov.
\newblock A method for solving the convex programming problem with convergence
  rate $o(1/k^2)$.
\newblock \emph{Dokl. Akad. Nauk SSSR}, 269:\penalty0 543--547, 1983.

\bibitem[Nesterov(2003)]{nesterov2003introductory}
Y.~Nesterov.
\newblock Introductory lectures on convex optimization. 2004, 2003.

\bibitem[Rakhlin and Sridharan(2013)]{rakhlin2013optimization}
S.~Rakhlin and K.~Sridharan.
\newblock Optimization, learning, and games with predictable sequences.
\newblock In \emph{Advances in Neural Information Processing Systems}, pages
  3066--3074, 2013.

\bibitem[Reddi et~al.(2018)Reddi, Kale, and Kumar]{j.2018on}
S.~J. Reddi, S.~Kale, and S.~Kumar.
\newblock On the convergence of adam and beyond.
\newblock In \emph{International Conference on Learning Representations}, 2018.

\bibitem[Tseng(2008)]{tseng2008accelerated}
P.~Tseng.
\newblock On accelerated proximal gradient methods for convex-concave
  optimization.
\newblock \emph{submitted to SIAM Journal on Optimization}, 01 2008.

\bibitem[Xiao(2010)]{xiao2010dual}
L.~Xiao.
\newblock Dual averaging methods for regularized stochastic learning and online
  optimization.
\newblock \emph{Journal of Machine Learning Research}, 11\penalty0
  (Oct):\penalty0 2543--2596, 2010.

\end{thebibliography}
\bibliographystyle{abbrvnat}

\iftrue
\clearpage
\appendix
 {
	%!TEX root = main.tex
%%%%%%%%%%%%%%%%%%%%%%%%%%%%%%%%%%%%%%%%%%%%%%%%%%%%%%%%%%%%%%%%%%%%%%%%%%%%%%%%%%%%%%%%%%%%%%
%%%%%%%%%%%%%%%%%%%%%%%%%%%%%%%%%%%%%%%%%%%%%%%%%%%%%%%%%%%%%%%%%%%%%%%%%%%%%%%%%%%%%%%%%%%%%%

\section{Proof of regret-to-rate conversion}

First, we discuss a generic scheme that enables us to relate our weighted regret bounds to optimality gap, hence the convergence rate. Once again, note that our analysis borrows tools and techniques from online learning literature and applies them to offline optimization setup. In essence, our conversion scheme applies to a special setting, where the convex loss is fixed across iterations. Let us give the respective Lemma and its proof.

%So far, we have dealt with regret analysis for the MOMD algorithm, however, it is essentially the convergence rate that we are ultimately interested in. We can go from regret bounds to convergence rates with a few algebraic manipulations. Without loss of generality, let's concentrate on the smooth case. We have shown that
%
%\[
%	\text{Regret := } \sum_{t=1}^{T} \alpha_t \ip{x_t - x_\ast}{\gradf{\bar{x}_t}} \leq \mathcal O \br{1} 
%\]
%
%Observe that all iterates $x_t$ are feasible, i.e., $x_T \in \mathcal K \, , \, \forall t$, which implies that $\bar{x}_t = \frac{\sum_{i=1}^{t} \alpha_i x_i }{\sum_{i=1}^{t} \alpha_t} \in \mathcal K$ as it is a convex combination of feasible iterates. Hence, $f(\bar{x}_t) - f^\ast \geq 0$ for all $t \in 1, ..., T$.
%
%We now claim that there exists a quantity that lower bounds the regret, such that it directly translates to accelerated optimization bounds.
%
%Next, we prove our claim and give the optimization bounds we aimed in the first place:

\begin{replemma}{lem:regret-to-rate}
	Consider weighted average $\bar{x}_t$ as in Eq.~\eqref{prelim:avg-iterates}. Let $R_T(x_\ast) = \sum_{t=1}^{T} \alpha_t \ip{x_t - x_\ast}{g_t} $ denote the weighted regret after T iterations, $\alpha_t = t$ and $g_t = \nabla f(\bar{x}_t)$. Then,
	
	\[
		f(\bar{x}_T) - f(x_\ast) \leq \frac{2 R_T(x_\ast)}{T^2}.
	\]
\end{replemma}

\begin{proof}

Let's define $A_t = \sum_{i=1}^{t} \alpha_i$. Then, by definition, we could express $x_t$ as

\begin{align} \label{def:A-t}
	x_t = \frac{A_t}{\alpha_t} \bar{x}_t - \frac{A_{t-1}}{\alpha_t} \bar{x}_{t-1}.
\end{align}

Then, use Eq.~\eqref{def:A-t} and replace $g_t$ by $\nabla f(\bar{x}_t)$ in the weighted regret expression, i.e.

\begin{align*}
	\sum_{t=1}^{T} \alpha_t \ip{x_t - x_\ast}{\gradf{\bar{x}_t}} &= \sum_{t=1}^{T} \alpha_t \ip{ \frac{A_t}{\alpha_t} \bar{x}_t - \frac{A_{t-1}}{\alpha_t} \bar{x}_{t-1} - x_\ast}{\gradf{\bar{x}_t} } \\
	&= \sum_{t=1}^{T} \alpha_t \ip{ \frac{A_t}{\alpha_t} \br{\bar{x}_t - x_\ast} - \frac{A_{t-1}}{\alpha_t} \br{\bar{x}_{t-1} - x_\ast}}{\gradf{\bar{x}_t} } \\
	&= \sum_{t=1}^{T} A_t \ip{ \bar{x}_t - x_\ast }{\gradf{\bar{x}_t} } - A_{t-1} \ip{\bar{x}_{t-1} - x_\ast}{\gradf{\bar{x}_t} }\\
	&= \sum_{t=1}^{T} \br{\sum_{i=1}^{t} \alpha_i \ip{ \bar{x}_t - x_\ast }{\gradf{\bar{x}_t} } } - \br{\sum_{i=1}^{t-1} \alpha_i \ip{\bar{x}_{t-1} - x_\ast}{\gradf{\bar{x}_t} }}\\
	&= \sum_{t=1}^{T} \alpha_t \ip{ \bar{x}_t - x_\ast }{\gradf{\bar{x}_t} } + \sum_{t=1}^{T}\sum_{i=1}^{t-1} \alpha_i \ip{\bar{x}_t - \bar{x}_{t-1}}{\gradf{\bar{x}_t}} \\
	&\geq \sum_{t=1}^{T} \alpha_t \br{ f(\bar{x}_t) - f(x_\ast) } + \sum_{t=1}^{T} \sum_{i=1}^{t-1} \alpha_i \br{ f(\bar{x}_t) - f( \bar{x}_{t-1} ) },
\end{align*}

where we used gradient inequality in the last line. We also take $\alpha_0 = 0$ and $A_0 = 0$. Then, we telescope the double summation and reorganize the terms

\begin{align*}
	&= \sum_{t=1}^{T} \alpha_t \br{ f(\bar{x}_t) - f(x_\ast) } + \sum_{t=1}^{T-1} \alpha_t \br{ f( \bar{x}_T ) - f( \bar{x}_t ) } \\
	& = \alpha_T \br{ f( \bar{x}_T ) - f( x_\ast ) } + \sum_{t=1}^{T-1} \alpha_t \br{ f(\bar{x}_t) - f(x_\ast) + f( \bar{x}_T ) - f( \bar{x}_t ) } \\
	&= \sum_{t=1}^{T} \alpha_t \br{ f( \bar{x}_T ) - f( x_\ast ) }.
\end{align*}

Having simplified the expression, we divide both sides by $A_T$ and conclude the proof. Observe that $A_T \geq \frac{T^2}{2}$, hence,

\begin{align*}
	\sum_{t=1}^{T} \alpha_t \br{ f( \bar{x}_T ) - f( x_\ast ) } &\leq \sum_{t=1}^{T} \alpha_t \ip{x_t - x_\ast}{\gradf{\bar{x}_t}} \\
	\frac{1}{A_T} \sum_{t=1}^{T} \alpha_t \br{ f( \bar{x}_T ) - f( x_\ast ) } &\leq \frac{1}{{A_T}} \sum_{t=1}^{T} \alpha_t \ip{x_t - x_\ast}{\gradf{\bar{x}_t}} \\
	f( \bar{x}_T ) - f( x_\ast ) &\leq \frac{2 R_T(x_\ast)}{T^2}.
\end{align*}

\end{proof}

\section{Proofs for the non-smooth setting}

As we have mentioned previously, for the weighted regret analysis in the non-smooth case, i.e., $f$ is only $G$-Lipschitz, please observe that we do not exploit the precise definitions of $g_t$ and $M_t$. As far as the regret analysis is concerned, their dual norm should be bounded. However, we especially rely on the fact that $g_t = \nabla f(\bar{x}_t)$ since it is necessary to obtain converge rates from regret-like bounds using Lemma~\ref{lem:regret-to-rate}.

Let us bring up the following relation which we will require for the regret analysis of both smooth and non-smooth objective. 

\begin{replemma}{lem:technical1} \label{replem:technical1}
	Let $\bc{a_i}_{i=1, ..., n}$ be a sequence of non negative numbers. Then, it holds that
	
	\[
		\sqrt{ \sum_{i=1}^{n} a_i } \leq \sum_{i=1}^{n} \frac{ a_i }{\sum_{j=1}^{i} a_j } \leq 2 \sqrt{ \sum_{i=1}^{n} a_i }.
	\]
	
\end{replemma}

Please refer to \citep{levy2018online, mcmahan2010adaptive} for the proof of Lemma~\ref{lem:technical1}, which is due to induction. We will also make use of the following bound (due to Young's Inequality)

\vspace{-3mm}
\begin{align} \label{def:gen-youngs-ineq}
	\alpha_t \norm{g_t - M_t}_\ast \norm{x_t - y_t} = \inf_{\rho > 0} \bc{ \frac{\rho}{2} \norm{g_t - M_t}_{\ast}^2 + \frac{\alpha_t^2}{2 \rho} \norm{x_t - y_t}^2 }.
\end{align}

\subsection{Proof of Theorem~\ref{thm:rate-nonsmooth-deterministic}}

\begin{reptheorem}{thm:rate-nonsmooth-deterministic} \label{repthm:rate-nonsmooth-deterministic}
	Consider the constrained optimization setting in Problem~(\ref{problem-def}), where $f: \mathcal K \rightarrow \mathbb R$ is a proper, convex and $G$-Lipschitz function defined over compact, convex set $\mathcal K$. Let $x^\ast \in \min_{x \in \mathcal K} f(x)$. Then, Algorithm~\ref{alg:UniXGrad} guarantees
	
	\vspace{-3mm}
	\begin{align}
		f( \bar{x}_T ) - \min_{x \in \mathcal K} f(x) \leq \frac{7 D \sqrt{ 1 + \sum_{t=1}^{T} \alpha_t^2 \norm{ g_t - M_t }_\ast^2 } - D}{T^2} \leq \frac{6 D}{T^2} + \frac{14 G D}{\sqrt{T}}.
	\end{align}
\end{reptheorem}

\begin{proof}
	\begin{align*}
		\sum_{t=1}^{T} \alpha_t \ip{x_t - x_\ast}{g_t} =& \sum_{t=1}^{T} \underbrace{ \alpha_t \ip{x_t - y_t}{g_t - M_t} }_\textrm{(A)} + \underbrace{ \alpha_t \ip{x_t - y_t}{M_t} }_\textrm{(B)} + \underbrace{ \alpha_t \ip{y_t - x^\ast}{g_t} }_\textrm{(C)}.
	\end{align*}
	
	\paragraph{Bounding (A)}
	\begin{align*}
			\sum_{t=1}^{T} \alpha_t \ip{x_t - y_t}{g_t - M_t} \leq& \sum_{t=1}^{T} \alpha_t \norm{g_t - M_t}_\ast \norm{x_t - y_t} \quad \text{(H{\"o}lder's Inequality)}\\
			\leq& \sum_{t=1}^{T} \frac{\rho}{2} \norm{g_t - M_t}_{\ast}^2 + \frac{\alpha_t^2}{2 \rho} \norm{x_t - y_t}^2 \quad \text{(Equation~(\ref{def:gen-youngs-ineq}))}.
	\end{align*}
		
	By setting $\rho = \alpha_t^2 \eta_{t+1}$, we get the following upper bound for term (A),
	\begin{align*}
			\sum_{t=1}^{T} \alpha_t \ip{x_t - y_t}{g_t - M_t} \leq \sum_{t=1}^{T} \frac{\alpha_t^2 \eta_{t+1}}{2} \norm{g_t - M_t}_{\ast}^2 + \frac{1}{2 \eta_{t+1}} \norm{x_t - y_t}^2
	\end{align*}	
	
	\paragraph{Bounding (B)}
	\begin{align*}
			\sum_{t=1}^{T} \alpha_t \ip{x_t - y_t}{M_t} \leq& \sum_{t=1}^{T} \frac{1}{\eta_t} \nabla_x D_{\mathcal R}(x_t, y_{t-1})^T (y_t - x_t) \quad (\text{Optimality for } x_t) \\
			=& \sum_{t=1}^{T} \frac{1}{\eta_t} \br{ D_{\mathcal R}(y_t, y_{t-1}) - D_{\mathcal R}(x_t, y_{t-1}) - D_{\mathcal R}(y_t, x_t) }.
	\end{align*}

	\paragraph{Bounding (C)}
	\begin{align*}
			\sum_{t=1}^{T} \alpha_t \ip{y_t - x^\ast}{g_t} \leq& \sum_{t=1}^{T} \frac{1}{\eta_t} \nabla_x D_{\mathcal R}(y_t, y_{t-1})^T (x^\ast - y_t) \quad (\text{Optimality for } y_t)\\
			=& \sum_{t=1}^{T} \frac{1}{\eta_t} \br{ D_{\mathcal R}(x^\ast, y_{t-1}) - D_{\mathcal R}(y_t, y_{t-1}) - D_{\mathcal R}(x^\ast, y_t) }.
	\end{align*}
	
	\paragraph{Final Bound}
	
	\begin{align*}
			\sum_{t=1}^{T} \alpha_t \ip{x_t - x_\ast}{g_t} \leq& \,  \sum_{t=1}^{T} \frac{\alpha_t^2 \eta_{t+1}}{2} \norm{g_t - M_t}_{\ast}^2 + \frac{1}{2 \eta_{t+1}} \norm{x_t - y_t}^2 \\
			&+ \frac{1}{\eta_t} \br{ D_{\mathcal R}(x^\ast, y_{t-1}) - D_{\mathcal R}(x^\ast, y_t) - D_{\mathcal R}(x_t, y_{t-1}) - D_{\mathcal R}(y_t, x_t) } \\
			\leq& \, \sum_{t=1}^{T} \frac{\alpha_t^2 \eta_{t+1}}{2} \norm{g_t - M_t}_{\ast}^2 + \frac{1}{2 \eta_{t+1}} \norm{x_t - y_t}^2 \\
			&+ \frac{1}{\eta_t} \br{ D_{\mathcal R}(x^\ast, y_{t-1}) - D_{\mathcal R}(x^\ast, y_t) - \frac{1}{2} \br{ \norm{x_t - y_{t}}^2 + \norm{x_t - y_{t-1}}^2 } } \\
			\leq& \, \sum_{t=1}^{T} \frac{\alpha_t^2 \eta_{t+1}}{2} \norm{g_t - M_t}_{\ast}^2 + \sum_{t=1}^{T-1} \br{ \frac{1}{\eta_{t+1}} - \frac{1}{\eta_t} } D_{\mathcal R}(x^\ast, y_t) \\
			& + \sum_{t=1}^{T} \br{ \frac{1}{\eta_{t+1}} - \frac{1}{\eta_t} } \norm{x_t - y_t}^2  + \frac{1}{\eta_1} D^2\\
			\leq& \, \sum_{t=1}^{T} \frac{\alpha_t^2 \eta_{t+1}}{2} \norm{g_t - M_t}_{\ast}^2 + \sum_{t=1}^{T} \br{ \frac{1}{\eta_{t+1}} - \frac{1}{\eta_t} } \norm{x_t - y_t}^2 + \frac{D^2}{\eta_T} + \frac{D}{2}  \\
			\leq& \, \sum_{t=1}^{T} \frac{\alpha_t^2 \eta_{t+1}}{2} \norm{g_t - M_t}_{\ast}^2 + D^2 \br{ \frac{2}{\eta_{T+1}} + \frac{1}{\eta_T} } + \frac{ D}{2} \\
			\leq& \, D \sum_{t=1}^{T} \frac{\alpha_t^2 \norm{ g_t - M_t }_\ast^2}{\sqrt{ 1 + \sum_{i=1}^{t} \alpha_t^2 \norm{g_t - M_t}_\ast^2 }} + \frac{3}{2} D \sqrt{ 1 + \sum_{t=1}^{T} \alpha_t^2 \norm{ g_t - M_t }_\ast^2 } + \frac{ D}{2} \\
			\leq& \, \frac{7}{2} D \sqrt{ 1 + \sum_{t=1}^{T} \alpha_t^2 \norm{ g_t - M_t }_\ast^2 } - \frac{D}{2}\\
			\leq& \, 3  D + 7G D \sqrt{ \sum_{t=1}^{T} \alpha_t^2} \\
			\leq& \, 3 D + 7G D T^{3/2}.
	\end{align*}
	
	We obtain the rate by applying Lemma~\ref{lem:regret-to-rate} to the weighted regret bound above.
	
\end{proof}

\subsection{Proof of Theorem~\ref{thm:rate-nonsmooth-stochastic}}

\begin{reptheorem}{thm:rate-nonsmooth-stochastic} \label{repthm:rate-nonsmooth-stochastic}
	Consider the optimization setting in Problem~(\ref{problem-def}), where $f$ is non-smooth, convex and $G$-Lipschitz. Let $\{ x_t \}_{t=1,..,T}$ be a sequence generated by Algorithm~\ref{alg:UniXGrad} such that $g_t = \tilde{\nabla} f(\bar{x}_t)$ and $M_t = \tilde{\nabla} f(\tilde{z}_t)$. With $\alpha_t = t$ and learning rate as in \eqref{prelim:learning-rate}, it holds that
	
	\[
		\mathbb E \left[ f(\bar{x}_T) \right] - \min\limits_{x \in \mathcal K} f(x) \leq \frac{6 D}{T^2} + \frac{14 G D}{\sqrt{T}}.
	\]
	
\end{reptheorem}

\begin{proof}
	Similar to $\nabla f(x) \leftrightarrow \tilde{\nabla} f(x)$ notation, $\tilde{g}_t$ denotes a stochastic but unbiased estimate of $g_t$ for any $t \in [0, .., T]$. Also note that $x^\ast \in \min_{x \in \mathcal K} f(x)$. We start with weighted regret bound,
	
	\[
		R_T(x_\ast) = \sum_{t=1}^{T} \alpha_t \ip{x_t - x^\ast}{g_t}.
	\]
	
	We separate $g_t$ as $\tilde{g}_t + (g_t - \tilde{g}_t)$ and re-write the above term as
	
	\[
		\sum_{t=1}^{T} \alpha_t \ip{x_t - x^\ast}{g_t} = \underbrace{ \sum_{t=1}^{T} \alpha_t \ip{x_t - x^\ast}{\tilde{g_t}} }_\textrm{(A)} + \underbrace{ \sum_{t=1}^{T} \alpha_t \ip{x_t - x^\ast}{g_t - \tilde{g_t}} }_\textrm{(B)}.
	\]
	
	Due to unbiasedness of the gradient estimates, expected value of $\alpha_t \ip{x_t - x^\ast}{g_t - \tilde{g_t}}$, conditioned on the average iterate $\bar{x}_t$ evaluates to 0. We will only need to bound the first summation whose analysis is identical to its deterministic counterpart up to replacing $g_t$ with $\tilde{g}_t$, and $M_t$ with $\tilde{M}_t$. Hence, term (A) is upper bounded by $6 D + 14 G D T^{3/2}$.
	
	In addition to the setup in the deterministic setting, we put forth the assumption that stochastic gradients have bounded norms, which is natural in the constrained optimization framework. Using Lemma~\ref{lem:regret-to-rate}, we translate the regret bound into the convergence rate, i.e,
	\[
	\mathbb E \left[ f(\bar{x}_T) \right] - \min\limits_{x \in \mathcal K} f(x) \leq \frac{6 D}{T^2} + \frac{14 G D}{\sqrt{T}}.
	\]
	
\end{proof}

\section{Proofs for the smooth setting}

We will now introduce an additional assumption that $f$ is $L$-smooth (see Eq.~\eqref{prelim:L-smooth}). In this section, we provide the weighted regret analysis for smooth functions in the presence of deterministic and stochastic oracles and convert these bound into suboptimality gap via our regret-to-rate scheme.

\subsection{Proof of Theorem~\ref{thm:rate-smooth-deterministic}}

\begin{reptheorem}{thm:rate-smooth-deterministic} \label{repthm:rate-smooth-deterministic}
	Consider the constrained optimization setting in Problem~(\ref{problem-def}), where $f: \mathcal K \rightarrow \mathbb R$ is a proper, convex and $L$-smooth function defined over compact, convex set $\mathcal K$. Let $x^\ast \in \min_{x \in \mathcal K} f(x)$. Then, Algorithm~\ref{alg:UniXGrad} ensures the following
		
	\vspace{-3mm}
	\begin{align}
		f(\bar{x}_T) - \min\limits_{x \in \mathcal K} f(x) \leq \frac{20 \sqrt{7} D^2 L}{T^2}.
	\end{align}
\end{reptheorem}

\begin{proof} 

Recall the regret analysis for the non-smooth, convex objective

\begin{align*}
	R_T(x_\ast) \leq& \; \frac{1}{2} \sum_{t=1}^{T} \eta_{t+1} \alpha_t^2 \norm{ g_t - M_t }_\ast^2 + \frac{1}{\eta_{t+1}} \norm{ x_t - y_{t} }^2 \\
	&+ \sum_{t=1}^{T} \frac{1}{\eta_t} \br{ D_{\mathcal R}(x^\ast, y_{t-1}) - D_{\mathcal R}(x^\ast, y_t) - \frac{1}{2} \br{ \norm{x_t - y_{t}}^2 + \norm{x_t - y_{t-1}}^2 } } \\
	\leq& \; \frac{1}{2} \sum_{t=1}^{T} \eta_{t+1} \alpha_t^2 \norm{ g_t - M_t }_\ast^2 + \frac{1}{2} \sum_{t=1}^{T} \br{ \frac{1}{\eta_{t+1}} - \frac{1}{\eta_t} } \norm{x_t - y_t}^2 + \sum_{t=1}^{T-1} \br{ \frac{1}{\eta_{t+1}} - \frac{1}{\eta_t} } D_{\mathcal R}(x^\ast, y_{t}) \\
	&- \frac{1}{2} \sum_{t=1}^{T} \frac{1}{\eta_t} \norm{x_t - y_{t-1}}^2 + \frac{D^2}{\eta_1} \\
	=& \; \frac{1}{2} \sum_{t=1}^{T} \eta_{t+1} \alpha_t^2 \norm{ g_t - M_t }_\ast^2 + \frac{1}{2} \sum_{t=1}^{T} \br{ \frac{1}{\eta_{t+1}} - \frac{1}{\eta_t} } \norm{x_t - y_t}^2 + \frac{1}{2} \sum_{t=1}^{T} \br{ \frac{1}{\eta_{t+1}} - \frac{1}{\eta_t} } \norm{x_t - y_{t-1}}^2\\
	&+ \sum_{t=1}^{T-1} \br{ \frac{1}{\eta_{t+1}} - \frac{1}{\eta_t} } D_{\mathcal R}(x^\ast, y_{t}) - \frac{1}{2} \sum_{t=1}^{T} \frac{1}{\eta_{t+1}} \norm{x_t - y_{t-1}}^2 + \frac{D^2}{\eta_1} \\
	\leq& \;  \frac{1}{2} \sum_{t=1}^{T} \eta_{t+1} \alpha_t^2 \norm{ g_t - M_t }_\ast^2 - \frac{1}{2} \sum_{t=1}^{T} \frac{1}{\eta_{t+1}} \norm{x_t - y_{t-1}}^2 + D^2 \br{ \frac{2}{\eta_{T+1}} + \frac{1}{\eta_T} + \frac{1}{\eta_1} }.
\end{align*}

The key challenge in this analysis is to exploit the negative term, i.e., $- \frac{1}{2} \sum_{t=1}^{T} \frac{1}{\eta_{t+1}} \norm{x_t - y_{t-1}}^2$, such that we could tighten the regret bound from non-smooth analysis. Using the smoothness of $f$ and that $\alpha_t = t$, $A_t = \sum_{i=1}^{t} \alpha_t$, $g_t = \gradf{ \xbar{t} }$ and $M_t = \gradf{ \ztilde{t} }$

\vspace{-3mm}

\begin{align*}
	\norm{g_t - M_t}_\ast^2 &\leq \frac{L^2 \alpha_t^2}{A_t^2} \norm{x_t - y_{t-1}}^2 \\
	&= \frac{4 L^2 t^2}{t^2 (t+1)^2} \norm{x_t - y_{t-1}}^2 \\
	&= \frac{4 L^2 }{\alpha_{t+1}^2} \norm{x_t - y_{t-1}}^2 \\
	&\leq \frac{4 L^2 }{\alpha_t^2} \norm{x_t - y_{t-1}}^2.
\end{align*}

\vspace{-3mm}

Hence,

\vspace{-3mm}

\begin{align*}
	- \frac{1}{\eta_{t+1}} \norm{x_t - y_{t-1}}^2 \leq - \frac{\alpha_t^2}{4 L^2 \eta_{t+1}} \norm{g_t - M_t}_\ast^2.
\end{align*}

After applying this upper bound and regrouping the terms we have

\begin{align*}
	R_T(x_\ast) &\leq \frac{1}{2} \sum_{t=1}^{T} \br{ \eta_{t+1} - \frac{1}{4L^2 \eta_{t+1}} } \alpha_t^2 \norm{g_t - M_t}_\ast^2 + D^2 \br{ \frac{2}{\eta_{T+1}} + \frac{1}{\eta_T} + \frac{1}{\eta_1} }.
\end{align*}

Define that $\tau^\ast = \max \bc{ t \in \bc{ 1, ..., T } \, : \, \frac{1}{\eta_{t+1}^2} \leq 7 L^2}$ such that $\forall t > \tau^\ast$, $\eta_{t+1} - \frac{1}{4 L^2 \eta_{t+1}} \leq -\frac{3}{4} \eta_{t+1} $. We can rewrite the above term as

\begin{align*}
	R_T(x_\ast) &\leq \frac{1}{2} \br{ \sum_{t=1}^{\tau^\ast} \br{ \eta_{t+1} - \frac{1}{4L^2 \eta_{t+1}} } \alpha_t^2 \norm{g_t - M_t}_\ast^2 + \sum_{t=\tau^\ast + 1}^{T} \br{ \eta_{t+1} - \frac{1}{4L^2 \eta_{t+1}} } \alpha_t^2 \norm{g_t - M_t}_\ast^2 }\\
	& \quad+ \frac{3 D^2}{\eta_{T+1}} + \frac{D^2}{\eta_1} \\
	& \leq \underbrace{ \frac{1}{2} \sum_{t=1}^{\tau^\ast} \eta_{t+1} \alpha_t^2 \norm{g_t - M_t}_\ast^2 + \frac{ D}{2}}_\textrm{(A)} + \underbrace{ \frac{3 D^2}{\eta_{T+1}} - \frac{3}{4} \sum_{t=\tau^\ast + 1}^{T} \eta_{t+1} \alpha_t^2 \norm{g_t - M_t}_\ast^2 }_\textrm{(B)}.
\end{align*}

\paragraph{Bounding (A): } We will simply need to use the definition of $\tau^\ast$ and Lemma~\ref{lem:technical1}

\begin{align*}
	\frac{1}{2} \sum_{t=1}^{\tau^\ast} \eta_{t+1} \alpha_t^2 \norm{g_t - M_t}_\ast^2 + \frac{ D}{2}&= D \sum_{t=1}^{\tau^\ast} \frac{ \alpha_t^2 \norm{g_t - M_t}_\ast^2 }{ \sqrt{ 1 + \sum_{i=1}^{t} \alpha_i^2 \norm{g_i - M_i}_\ast^2 } } + \frac{ D}{2}\\
	& \leq 2 D \sqrt{ 1 + \sum_{t=1}^{\tau^\ast} \alpha_t^2 \norm{g_t - M_t}_\ast^2 } \\
	&= \frac{4 D^2}{\eta_{\tau^\ast + 1}} \\
	&\leq 4 \sqrt{7} D^2 L.
\end{align*}

\paragraph{Bounding (B): }

\begin{align*}
	\text{(B)} &\leq \frac{3 D}{2} \br{  \sqrt{ 1 + \sum_{t=1}^{T} \alpha_t^2 \norm{ g_t - M_t } _\ast^2 } - \sum_{t=\tau^\ast + 1}^{T} \frac{\alpha_t^2 \norm{g_t - M_t}_\ast^2}{\sqrt{ 1 + \sum_{i=1}^{t} \alpha_i^2 \norm{g_i - M_i}_\ast^2 }} } \\
	&\leq  \frac{3 D }{2} + \frac{3 D}{2} \br{ \sum_{t=1}^{T} \frac{\alpha_t^2 \norm{g_t - M_t}_\ast^2}{\sqrt{ 1 + \sum_{i=1}^{t} \alpha_i^2 \norm{g_i - M_i}_\ast^2 }}   - \sum_{t=\tau^\ast + 1}^{T} \frac{\alpha_t^2 \norm{g_t - M_t}_\ast^2}{\sqrt{ 1 + \sum_{i=1}^{t} \alpha_i^2 \norm{g_i - M_i}_\ast^2 }} } \\
	&\leq \frac{3 D }{2} + \frac{3 D}{2} \sum_{t=1}^{\tau^\ast} \frac{\alpha_t^2 \norm{g_t - M_t}_\ast^2}{\sqrt{ 1 + \sum_{i=1}^{t} \alpha_i^2 \norm{g_i - M_i}_\ast^2 }} \\
	&\leq 3 D \sqrt{ 1 + \sum_{t=1}^{\tau^\ast} \alpha_i^2 \norm{g_i - M_i}_\ast^2 } \\
	&= \frac{6 D^2}{\eta_{\tau^\ast + 1}} \\
	&\leq 6 \sqrt{7} D^2 L.
\end{align*}

\paragraph{Final Bound: } What remains is to simply bring the term (A) and (B) together.

\begin{align*}
	R_T(x_\ast) &\leq \frac{1}{2} \sum_{t=1}^{\tau^\ast} \eta_{t+1} \alpha_t^2 \norm{g_t - M_t}_\ast^2  + \frac{ D}{2} + \frac{3 D^2}{\eta_{T+1}} - \frac{3}{4} \sum_{t=\tau^\ast + 1}^{T} \eta_{t+1} \alpha_t^2 \norm{g_t - M_t}_\ast^2 \\
	&\leq 10 \sqrt{7} D^2 L.
\end{align*}

We conclude the proof by applying Lemma~\ref{lem:regret-to-rate} and get $f(\bar{x}_T) - \min\limits_{x \in \mathcal K} f(x) \leq \frac{20 \sqrt{7} D^2 L}{T^2}$.

\end{proof}

\subsection{Proof of Theorem~\ref{thm:rate-smooth-stochastic}}

In this setting, we will make an additional, but classical, bounded variance assumption on the stochastic gradient oracles. Recall the bounded variance assumption in Eq. \eqref{def:bounded-variance} and let us define $\xi_t = ( \tilde{g}_t - \tilde{M}_t ) - \left( g_t - M_t \right)$. Since $\| \xi_t \|_\ast^2 \leq 2\| \tilde{g}_t - g_t \|_\ast^2 + 2 \| \tilde{M}_t - M_t \|_\ast^2$, we can write,
\begin{align} \label{def:bounded-variance-xi}
	\mathbb E \left[ \| \xi_t\|_\ast^2 | \bar{x}_t \right] \leq 4 \sigma^2.
\end{align}

Next, we will present our final convergence theorem.

\begin{reptheorem}{thm:rate-smooth-stochastic}\label{repthm:rate-smooth-stochastic}
	Consider the optimization setting in Problem~(\ref{problem-def}), where $f$ is $L$-smooth and convex. Let $\{ x_t \}_{t=1,..,T}$ be a sequence generated by Algorithm~\ref{alg:UniXGrad} such that $g_t = \tilde{\nabla} f(\bar{x}_t)$ and $M_t = \tilde{\nabla} f(\tilde{z}_t)$. With $\alpha_t = t$ and learning rate as in (\ref{prelim:learning-rate}), it holds that
	
	\[
		\mathbb E \left[ f(\bar{x}_T) \right] - \min\limits_{x \in \mathcal K} f(x) \leq \frac{224 \sqrt{14} D^2 L}{T^2} + \frac{14\sqrt{2} \sigma D }{\sqrt{ T}}.
	\]
	
\end{reptheorem}

\begin{proof}
	We start out with weighted regret, the same way as in Theorem~\ref{thm:rate-nonsmooth-stochastic}
	
	\[
		\sum_{t=1}^{T} \alpha_t \ip{x_t - x^\ast}{g_t} \leq \underbrace{ \sum_{t=1}^{T} \alpha_t \ip{x_t - x^\ast}{\tilde{g_t}} }_\textrm{(A)} + \underbrace{ \sum_{t=1}^{T} \alpha_t \ip{x_t - x^\ast}{g_t - \tilde{g_t}} }_\textrm{(B)}.
	\]
	
	We already know that term (B) is zero in expectation. Following the proof steps of Theorem~\ref{thm:rate-nonsmooth-stochastic}, we could upper bound term (A) as
	
	\begin{align*}
		&\leq  \frac{1}{2} \sum_{t=1}^{T} \eta_{t+1} \alpha_t^2 \| \tilde{g}_t - \tilde{M}_t \|_\ast^2 - \frac{1}{2} \sum_{t=1}^{T} \frac{1}{\eta_{t+1}} \norm{x_t - y_{t-1}}^2 + D^2 \br{ \frac{3}{\eta_{T+1}} + \frac{1}{\eta_1} } \\
		&= \frac{ D}{2} + D \sum_{t=1}^{T} \frac{ \alpha_t^2 \| \tilde{g}_t - \tilde{M}_t \|_\ast^2}{\sqrt{1 + \sum_{i=1}^{t} \alpha_i^2 \| \tilde{g}_t - \tilde{M}_t \|_\ast^2 }} + \frac{3 D}{2} \sqrt{ 1 + \sum_{t=1}^{T} \alpha_t^2 \| \tilde{g}_t - \tilde{M}_t \|_\ast^2} - \sum_{t=1}^{T} \frac{\norm{x_t - y_{t-1}}^2}{2 \eta_{t+1}} \\
		&\leq  \frac{7 D}{2} \sqrt{ 1 + \sum_{t=1}^{T} \alpha_t^2 \| \tilde{g}_t - \tilde{M}_t \|_\ast^2 } - \frac{1}{2} \sum_{t=1}^{T} \frac{1}{\eta_{t+1}} \norm{x_t - y_{t-1}}^2 .
		\vspace{4mm}
	\end{align*}
	Now lets denote,
	
	$$
	B_t^2: = \min \{\|g_t-M_t\|_\ast^2, \| \tg_t -\tM_t\|_\ast^2\},
	$$
	as well as an auxiliary learning rate which we will only use for the analysis
	
	\vspace{-2mm}
	\begin{align} \label{prelim:learning-rate-imaginary}
		\teta_t = \frac{2 D}{\sqrt{ 1 + \sum\limits_{i=1}^{t-1} \alpha_i^2 B_i^2} }.
	\end{align}
	Clearly, for any $t\in[T]$ we have ${1}/{\teta_t} \leq {1}/{\eta_t}$, and therefore,
       
        \begin{align} \label{eq:SmoothStochastic1}
        		-\frac{1}{\eta_{t+1}} \norm{g_t - M_t}_\ast^2 &\leq  -\frac{1}{\teta_{t+1}}B_t^2.
        \end{align}
        Also, for $\xi_t = (\tilde{g}_t - \tilde{M}_t) - (g_t - M_t)$, we can write,
        
        \begin{align}\label{eq:Ineq33}
        \|\tg_t -\tM_t\|_\ast^2 \leq 2\|g_t-M_t\|_\ast^2  +2\|\xi_t\|_\ast^2.
        \end{align}
        Thus,
        \begin{align*}
        \|\tg_t -\tM_t\|_\ast^2 
        &= 
        B_t^2 + \left( \|\tg_t -\tM_t\|_\ast^2  -  \min \{\|g_t-M_t\|_\ast^2, \| \tg_t -\tM_t\|_\ast^2\} \right)\\
        & =
        B_t^2  + \max\{ 0,  \|\tg_t -\tM_t\|_\ast^2- \| g_t -M_t\|_\ast^2\} \\
        &\leq
        B_t^2 + B_t^2 + 2\|\xi_t\|_\ast^2 \\
        & = 
        2B_t^2 + 2\|\xi_t\|_\ast^2,
        \end{align*}
        where the last inequality is due to the fact that if $\|\tg_t -\tM_t\|_\ast^2\geq \| g_t -M_t\|_\ast^2$, then 
        $B_t^2: =\| g_t -M_t\|_\ast^2$. Then, we combine this with Eq.~\eqref{eq:Ineq33} to deduce that $\|\tg_t -\tM_t\|_\ast^2 - \|g_t -M_t\|_\ast^2 \leq B_t^2 + 2\|\xi_t\|_\ast^2$.
        
        We will take conditional expectation after we simplify the expression. Now, we plug Eq.~\eqref{eq:SmoothStochastic1} and \eqref{eq:Ineq33} into above bound,
        
	\begin{align*}
		&\leq \frac{7 D}{2} \sqrt{ 1 + 2 \sum_{t=1}^{T} \alpha_t^2 B_t^2 + \alpha_t^2 \norm{\xi_t}_\ast^2 } - \frac{1}{2} \sum_{t=1}^{T} \frac{1}{4 L^2 \teta_{t+1}} \alpha_t^2 B_t^2 \\
		&\leq \frac{7 D}{\sqrt{2}} \sqrt{ \sum_{t=1}^{T} \alpha_t^2 \norm{\xi_t}_\ast^2 } + \frac{7 D}{2} \sqrt{ 1 + 2 \sum_{t=1}^{T} \alpha_t^2 B_t^2 } - \frac{1}{2} \sum_{t=1}^{T} \frac{1}{4 L^2 \teta_{t+1}} \alpha_t^2 B_t^2\\
		&\leq \frac{7  D}{2} + \frac{7 D}{\sqrt{2}} \sqrt{ \sum_{t=1}^{T} \alpha_t^2 \norm{\xi_t}_\ast^2 } + 7 D\sum_{t=1}^{T} \frac{\alpha_t^2 B_t^2}{\sqrt{ 1 + 2 \sum_{i=1}^{t} \alpha_i^2 B_i^2 }} - \frac{1}{2} \sum_{t=1}^{T} \frac{1}{4 L^2 \teta_{t+1}} \alpha_t^2 B_t^2 \\
		&\leq  \frac{7  D}{2} + \frac{7 D}{\sqrt{2}} \sqrt{ \sum_{t=1}^{T} \alpha_t^2 \norm{\xi_t}_\ast^2 } + 7 D\sum_{t=1}^{T} \frac{\alpha_t^2 B_t^2}{\sqrt{ 1 + \sum_{i=1}^{t} \alpha_i^2 B_i^2 }} - \frac{1}{2} \sum_{t=1}^{T} \frac{1}{4 L^2 \teta_{t+1}} \alpha_t^2 B_t^2\\
		&\leq \underbrace{ \frac{7}{2} \sum_{t=1}^{T} \left( \teta_{t+1} - \frac{1}{28 L^2 \teta_{t+1}} \right) \alpha_t^2 B_t^2 + \frac{7 D}{2} }_\textrm{(A)} + \underbrace{ \frac{7 D}{\sqrt{2}} \sqrt{ \sum_{t=1}^{T} \alpha_t^2 \norm{\xi_t}_\ast^2 }  }_\textrm{(B)}.
	\end{align*}
	
	\paragraph{Bounding (A): } We will make use of the exact same approach as we did in Theorem~\ref{thm:rate-smooth-deterministic}, where we defined an auxiliary time variable $\tau^\ast$ to characterize the behavior of the learning rate.
	
	Now, let us denote $\tau^\ast = \max \bc{ t \in \bc{ 1, ..., T } \, : \, \frac{1}{\teta_{t+1}^2} \leq 56 L^2}$. It implies that
	
	\begin{align} \label{eq:tau-star}
		\teta_{t+1} - \frac{1}{28 L^2 \teta_{t+1}} \leq - \teta_{t+1}, \quad \forall t > \tau^\ast.
	\end{align}
	
	Then, we could proceed as
	
	\begin{align*}
		\text{(A)} &= \frac{7}{2} \sum_{t=1}^{\tau^\ast} \left( \teta_{t+1} - \frac{1}{28 L^2 \teta_{t+1}} \right) \alpha_t^2 B_t^2 + \frac{7}{2} \sum_{t=\tau^\ast + 1}^{T} \left( \teta_{t+1} - \frac{1}{28 L^2 \teta_{t+1}} \right) \alpha_t^2 B_t^2 + \frac{7  D}{2} \\
		&\leq \frac{7}{2} \sum_{t=1}^{\tau^\ast} \teta_{t+1} \alpha_t^2 B_t^2 - \frac{7}{2} \sum_{t=\tau^\ast + 1}^{T} \teta_{t+1} \alpha_t^2 B_t^2 + \frac{7  D}{2}\\
		&\leq \frac{7}{2} \sum_{t=1}^{\tau^\ast} \teta_{t+1} \alpha_t^2 B_t^2 + \frac{7  D}{2}\\
		&= 7 D \sum_{t=1}^{\tau^\ast} \frac{\alpha_t^2 B_t^2}{ \sqrt{ 1 + \sum_{i=1}^{t} \alpha_i^2 B_i^2 } } + \frac{7  D}{2}\\
		&\leq 14 D \sqrt{ 1 + \sum_{t=1}^{\tau^\ast} \alpha_t^2 B_t^2 } \\
		&\leq \frac{28 D^2}{\teta_{\tau^\ast+1}}  \\
		&\leq 112 \sqrt{14} D^2 L.
	\end{align*}

	\paragraph{Bounding (B): } Following bounded variance definition in Eq.~\eqref{def:bounded-variance}, we can write $\mathbb E [ \| \xi_t \|_\ast^2 ] \leq 4 \sigma^2$. After taking expected value conditioned on $\bar{x}_t$, we simply use Jensen's inequality to complete the proof
	
	\begin{align*}
		\mathbb E \left[ \frac{7 D}{\sqrt{2}} \sqrt{ \sum_{t=1}^{T} \alpha_t^2 \norm{\xi_t}_\ast^2 } \, \Bigg | \, \bar{x}_t \right] &\leq \frac{7 D}{\sqrt{2}} \sqrt{ \mathbb E \left[\sum_{t=1}^{T} \alpha_t^2 \norm{\xi_t}_\ast^2 \, \bigg| \, \bar{x}_t \right] } \\
		&= \frac{7 D}{\sqrt{2}} \sqrt{ \sum_{t=1}^{T} \alpha_t^2 \mathbb E \left[ \norm{\xi_t}_\ast^2 \, \big| \, \bar{x}_t \right] } \\
		&\leq \frac{7 D}{\sqrt{2}} \sqrt{ \sum_{t=1}^{T} 4 \alpha_t^2  \sigma^2 } \\
		&\leq \frac{14 D \sigma}{\sqrt{2}} \sqrt{ T^3 } \\
		&= \frac{14 \sigma D T^{3/2}}{\sqrt{2}}.
	\end{align*}
	
	Finally, we combine all these bounds together and feed them through Lemma~\ref{lem:regret-to-rate} to obtain the final rate.
	
	\begin{align*}
		\mathbb E \left[ f(\bar{x}_T) \right] - \min\limits_{x \in \mathcal K} f(x) &\leq \frac{224 \sqrt{14} D^2 L}{T^2} + \frac{14\sqrt{2} \sigma D }{\sqrt{ T}}.
	\end{align*}

\end{proof}

}
\fi

\end{document}